\documentclass[11pt,epsf]{article}
\usepackage{graphicx}
\usepackage{amsthm}
\usepackage{amsfonts}
\usepackage{amssymb}
\title{Some analytic properties of the partial theta function}
\author{Vladimir Petrov Kostov\\ 
Universit\'e C\^ote d'Azur, CNRS, LJAD, France\\   
e-mail: vladimir.kostov@unice.fr} 
\date{}
\newtheorem{tm}{Theorem}
\newtheorem{defi}[tm]{Definition}
\newtheorem{rem}[tm]{Remark}
\newtheorem{rems}[tm]{Remarks}
\newtheorem{lm}[tm]{Lemma}

\newtheorem{prop}[tm]{Proposition}
\newtheorem{nota}[tm]{Notation}
\begin{document}
\maketitle
\begin{abstract}
  We prove new properties of the zero set of
  Ramanujan's partial theta function
  $\theta (q,x):=\sum _{j=0}^{\infty}q^{j(j+1)/2}x^j$, $q\in (-1,0)\cup (0,1)$,
  $x\in \mathbb{R}$. We show that for each $q\in (0,1)$, there exists a line
  Re$x=-a$, $a\geq 5$, such that all real zeros of $\theta(q,.)$
  lie to its left and all complex zeros to its right. A similar property is
  proved for $q\in (-1,0)$. For $q\in (0,1)$, there are no real zeros
  $\geq -6$. For $q\in (-1,0)$, there are no negative zeros $\geq -2.4$ and
  no positive zeros $\leq 2.4$, except the smallest one. \\ 

{\bf Key words:} partial theta function, Jacobi theta function, 
Jacobi triple product\\ 

{\bf AMS classification:} 26A06~~~\, \, \, UDK:~ 517.531.55
\end{abstract}

\section{Introduction}

The present paper considers the {\em partial theta function} which is the sum
of the series
$$\theta (q,x):=\sum _{j=0}^{\infty}q^{j(j+1)/2}x^j~.$$
We treat
$q$ as a parameter and $x$ as a variable. For each $q$ fixed, $|q|<1$, the
function $\theta$ is an entire function in~$x$. The partial theta function
satisfies the functional equation 

\begin{equation}\label{equqx}
\theta (q,x)=1+qx\theta (q,qx)~.
\end{equation}
The function $\theta$  has been studied
in the complex situation, i.e., when $(q,x)\in \mathbb{C}^2$, $|q|<1$,
see \cite{KoAM}, and in the real one which can be subdivided between the cases

$$A)~~~(q,x)\in (0,1)\times \mathbb{R}~~~\, \, {\rm and}~~~\, \, B)~~~
(q,x)\in (-1,0)\times \mathbb{R}$$
(one trivially has $\theta (0,x)\equiv 1$).
We formulate and prove new analytic properties of $\theta$ in these cases and we
complete and improve certain previous results.

The partial theta function finds applications in various domains. When
pure mathematics is concerned, one should mention Ramanujan type $q$-series
(see~\cite{Wa}), asymptotic analysis (see~\cite{BeKi}) and the theory of (mock)
modular forms (see \cite{BrFoRh}). There are situations when the
function is equally interesting for mathematicians and physicists. Thus
the paper \cite{So} speaks about its role in statistical physics
and combinatorics, the articles \cite{BFM} and \cite{CMW}
explain its applications to the research of problems about asymptotics
and modularity of partial and false theta functions and their
interaction with representation theory and conformal field theory.
In \cite{WPG} this function appears in the context of
quantum many-body systems.

The famous Indian mathematician Srinivasa Ramanujan has studied
the partial theta function in his lost hotebook,
see~\cite{AnBe} and~\cite{Wa}. 
Information about Appell-Lerch sums and mock theta functions
can be found in~\cite{EM}. The link between 
$\theta$ and Artin-Tits monoids is revealed in~\cite{FGM}; Pad\'e
approximants of $\theta$ are considered in~\cite{LuSa}.
Andrews-Warnaar identities for
the partial theta function are explored in~\cite{WM}, \cite{Wei}
and~\cite{Sun}.
In \cite{Pr} one can find an explicit combinatorial interpretation of the
coefficients of the leading root of $\theta$ as a series in~$q$.

Recently, the connection of $\theta$ to section-hyperbolic polynomials has
motivated a renewed interest in its analytic properties. The mentioned
polynomials have positive coefficients, all their roots are real negative and
all their finite sections (i.~e. truncations) have also all their roots
real negative. Classical works in this domain belong to
Hardy, Petrovitch and Hutchinson (see~\cite{Ha}, \cite{Pe} and~\cite{Hu})
and this activity has been continued in the more recent papers
\cite{Ost}, \cite{KaLoVi} and \cite{KoSh}. The author of the present lines
has consecrated the articles \cite{KoFAA}--\cite{KoAnn24} to the study
of the analytic properties of $\theta$ which are interesting in their own.
In particular, pictures of the zero set of $\theta$ can be found
in~\cite{KoSe}.

\section{The new results}

\subsection{Location of the zeros of $\theta$\protect\label{subsecloc}}

We begin this subsection with results concerning the bounds for the real zeros
of $\theta (q,.)$. We remind that in case A) this function has no
non-negative zeros.

\begin{rem}\label{remKatsnelson}
  {\rm In what follows we allow in certain situations the values $1$ and $-1$ 
    for the parameter $q$. This is because a result of V.~Katsnelson,
    see \cite{Ka}, implies that the series of $\theta$ converges to $1/(1-x)$
    as $q\rightarrow 1^-$ uniformly on compact sets contained inside
    the contour $K\subset \mathbb{C}$}

  $$K~:~x=e^{t\pm it}~~~\, {\rm or~equivalently}~~~\,
  \mathbb{R}^2\ni (\xi ,\eta )=(e^t\cos t,~\pm e^t\sin t)~,~~~\,
  t\in [0,\pi ]~.$$
  {\rm The domain bounded by this contour is relevant in case A).
    It contains the unit disk
    and the segment
    $\{ y=0,~x\in [-e^{\pi},1]\}$. To
    consider also case B) we need the following equality:}
  
\begin{equation}\label{equdecomp}
    \theta (q,x)=\theta_1+qx\theta_2~,~~~\,
    \theta_1(q,x):=\theta (q^4,x^2/q)~,~~~\,
  \theta_2(q,x):=\theta (q^4,qx^2)~.\end{equation}
    {\rm The analog of the contour $K$ for $q\in (-1,0)$ is obtained by setting
      $-x^2=e^{t\pm it}$, i.~e.}

    $$K^{\bullet}~:~x=\pm e^{t/2+i(\pi \pm t)/2}~~~\, {\rm or}~~~\,
    (\xi ,\eta )=\pm (e^{t/2}\cos ((\pi \pm t)/2),~e^{t/2}\sin ((\pi \pm t)/2))~,
    ~~~\, t\in [0,\pi ]~,$$
      {\rm because as $q\rightarrow -1^+$,
        one gets $x^2q^{\pm 1}\rightarrow -x^2$. Thus for $q\rightarrow -1^+$,
        the series of $\theta$ converges to $(1-x)/(1+x^2)$ uniformly
        on compact sets inside the contour~$K^{\bullet}$. The contours $K$
      and $K^{\bullet}$ consist of $2$ and $4$ arcs of logarithmic spirals  respectively.}
\end{rem}

For the upper bound of the real zeros in case A),
the following proposition
holds true:

\begin{prop}\label{prop6}
  (1) For $q\in (0,1)$, the partial theta function has no real zeros $\geq -6$.

  (2) The partial theta function has no real zeros for
  $(q,x)\in \tilde{Q}:=[0.4,1]\times [-10.5,0]$.
  Its values on $\partial \tilde{Q}$ are $\geq 0.0049$. 
\end{prop}

Numerical computation suggests that $\theta (0.265,.)$ has a zero in the
interval $(-6.1,-6)$. In this sense part (1) of the proposition is
sufficiently sharp. The analog of Proposition~\ref{prop6}
in case B) is formulated as follows: 

\begin{prop}\label{prop2.2}
  (1) For $q\in (-1,0)$, any negative zero of $\theta (q,.)$ is $<-2.4$.
  \vspace{1mm}
  
  (2) For $q\in (-1,0)$, any positive zero of $\theta (q,.)$, except
  the smallest one, is $>2.4$.

  (3) The partial theta function has no negative zero for
  $(q,x)\in \tilde{Q}_-:=[-1,-0.75]\times [-3.1,0]$. For
  $(q,x)\in \tilde{Q}_+:=[-1,-0.8]\times [0,3.2]$, its only positive zero
  is the smallest one. One has $\theta >0.0049$ for
  $(q,x)\in \partial \tilde{Q}_-$ and 
  $\theta <-0.015$ for $x=3.2$, $q\in [-1,-0.78]$.
  The three smallest positive zeros of
  $\theta (-0.78,.)$ are $<3.2$. 
  
\end{prop}

Propositions~\ref{prop6} and \ref{prop2.2} are proved in
Section~\ref{secprprop2.2}.

\begin{rems}

  {\rm (1) The choice of the sets in the formulation of
    Proposition~\ref{prop2.2} is clarified in part~(2)
    of Remarks~\ref{remsLP}.
    \vspace{1mm}
    
    (2) Numerical computation shows that $\theta (-0.7,.)$
    (resp. $\theta (-0.78,.)$)
    has a zero in the interval $(-2.7,-2.6)$ (resp. in $(2.7,2.8)$).}
\end{rems}
 
We remind that the {\em spectrum} $\Gamma$
of $\theta$ is the set of values of the
parameter $q$ for which $\theta (q,.)$ has a multiple zero. This notion was
introduced by Boris Shapiro in~\cite{KoSh}. We list some facts about the
spectrum of~$\theta$. They correspond to \cite[Theorem~1]{KoBSM1}
(see 1.--3.), 
\cite[Theorem~1.4]{KoPRSE2} (see 4.--7.) and \cite[Theorem~8]{KoSe} (see~8.).
\vspace{1mm}

{\em Properties of the spectrum of~$\theta$:}

\begin{enumerate}
\item
For $q\in (0,1)$, the spectrum consists of countably-many values of $q$
denoted by $0<\tilde{q}_1<\tilde{q}_2<\cdots$, where
$\lim_{j\rightarrow \infty}\tilde{q}_j=1^-$. We set $\tilde{q}_0:=0$.
One has $\tilde{q}_1=0.3092493386\ldots$, see~\cite{KoSh}. For
$q\in (0,\tilde{q}_1)$, all zeros of $\theta$ are negative and distinct:
$\cdots <\xi_3<\xi_2<\xi_1<0$. 

\item
  For $\tilde{q}_N\in \Gamma$, the function $\theta (\tilde{q}_N,.)$
  has exactly one multiple real zero which is of multiplicity $2$ and
  is the rightmost of its real zeros. The real zeros of $\theta$ are:
  $\cdots <\xi_{2N+2}<\xi_{2N+1}<\xi_{2N}=\xi_{2N-1}<0$. 

\item
  For $q\in (\tilde{q}_N,\tilde{q}_{N+1})$, the function $\theta$ has exactly
  $N$ complex conjugate pairs of zeros (counted with multiplicity).
  Thus when the parameter $q$ increases in $(0,1)$ and passes
  through a spectral value,
  the function $\theta$ loses two real zeros and acquires
  a complex conjugate pair. For no value of $q$ do two complex
  conjugate zeros coalesce
  to become two real zeros. 

\item
  For $q\in (-1,0)$, there exists a sequence of values of $q$ (denoted by
  $\bar{q}_j$) tending to $-1^+$ such that $\theta (\bar{q}_k,.)$
  has a double real zero $\bar{y}_k$
  (the rest of its real zeros being simple). For the remaining values of
  $q\in (-1,0)$, the function $\theta (q,.)$ has no multiple
  real zero. One has $\bar{q}_1=0.72713332\ldots$, 
  $\bar{q}_2=0.78374209\ldots$ and $\bar{q}_3=0.84160192\ldots$,
  see~\cite{KoPRSE2}.  

\item
  For $k$ odd (respectively, for $k$ even), one has $\bar{y}_k<0$,
  $\theta (\bar{q}_k,.)$ has a local minimum
  at $\bar{y}_k$ and $\bar{y}_k$ is the rightmost of the real negative zeros
  of $\theta (\bar{q}_k,.)$ (respectively, $\bar{y}_k>0$,
  $\theta (\bar{q}_k,.)$ has a local maximum at $\bar{y}_k$ and for $k$
  sufficiently large, $\bar{y}_k$ is the
  leftmost but one (second from the left) of the real negative zeros of
  $\theta (\bar{q}_k,.)$).

\item
  For $k$ sufficiently large, one has $-1<\bar{q}_{k+1}<\bar{q}_k<0$.

\item
  For $k$ sufficiently large and for $q\in (\bar{q}_{k+1},\bar{q}_k)$,
  the function $\theta (q,.)$ has exactly
  $k$ complex conjugate pairs of zeros counted with multiplicity. When $q$
  decreases in $(-1,0)$ and passes through a spectral value, two real zeros
  of $\theta$ coalesce to form a complex conjugate pair. Complex zeros do
  not coalesce to give birth to real zeros.

\item
  For $q\in (-1,0)$, no zero of $\theta$ crosses
  the imaginary axis.
\end{enumerate}

\begin{rems}\label{remsLP}
  {\rm (1) For $q\in (0,\tilde{q}_1]$ and $q\in [\bar{q}_1,0)$, the partial
    theta function belongs to the Laguerre-P\'olya classes $\mathcal{LP}I$
    and $\mathcal{LP}$ respectively; it is of order~$0$. For
    $q\in (\tilde{q}_1,1)$ and $q\in (-1,\bar{q}_1)$, it is the product of
    such a function and a real polynomial in $x$ without real roots.
\vspace{1mm}

(2) In the formulation of part (3) of Proposition~\ref{prop2.2}
the segment $[-1,-0.78]$ is longer than the segment $[-1,-0.8]$
  in order to prove the last statement of part (3). The number $-0.78$ is
  chosen close to the spectral number $\bar{q}_2=-0.7837\ldots$.}
\end{rems}

Complex conjugate pairs do not go too close to the origin
or too far from it.
Concretely, the best
results known to-date about the location of the complex conjugate pairs
read:

\begin{tm}\label{tm4parts}(1) {\rm (\cite[Theorem~1]{KoArxiv})}
  For $q\in (0,1)$,
  the complex conjugate pairs with non-negative real part (if any) of
  $\theta (q,.)$ belong to the half-annulus
  $\mathcal{A}:=\{ {\rm Re}x\geq 0,~1<|x|<5\}$.
  \vspace{1mm}
  
  (2) {\rm (\cite[Theorem~3]{KoArxiv})} For $q\in (0,1)$, the complex
  conjugate pairs
  of zeros of $\theta (q,.)$ with negative real part
belong to the left open half-disk of radius $49.8$ centered at the origin.
\vspace{1mm}

(3) {\rm (\cite[Theorem~1]{KoMatStud})} For any fixed $q\in (0,1)$,
the partial theta function has no zeros in the domain
$\mathcal{D}:=\{ x\in \mathbb{C}:|x|\leq 3,~{\rm Re} x\leq 0,
|{\rm Im} x|\leq 3/\sqrt{2}\}$ (with $3/\sqrt{2}=2.121320344\ldots$).
\vspace{1mm}

(4) {\rm (\cite[Theorem~1]{KoAnn24})} For each $q\in (-1,0)\cup (0,1)$ fixed,
the function $\theta$ has no zeros in the closed unit disk
$\overline{\mathbb{D}_1}$.
\vspace{1mm}

(5) {\rm (\cite[part~(2) of Theorem~1]{KoFAA19})} For $q\in (-1,0)$, all complex conjugate pairs of zeros
of $\theta$ belong to the rectangle $\{ x:|{\rm Re}x|<364.2,|{\rm Im}x|<132\}$.
\end{tm}

%
%
%

\subsection{Improvement of \protect\cite[Theorem~1.4]{KoPRSE2}}

In the present subsection we improve \cite[Theorem~1.4]{KoPRSE2} (see in 
Subsection~\ref{subsecloc} properties 5.--7. of the spectrum of $\theta$),
by getting rid of the condition ``for $k$
sufficiently large''.

\begin{tm}\label{tmimprove}
(1) For $k$ odd (respectively, for $k$ even), one has $\bar{y}_k<0$,
  $\theta (\bar{q}_k,.)$ has a local minimum
  at $\bar{y}_k$ and $\bar{y}_k$ is the rightmost of the real negative zeros
  of $\theta (\bar{q}_k,.)$ (respectively, $\bar{y}_k>0$,
  $\theta (\bar{q}_k,.)$ has a local maximum at $\bar{y}_k$
  and $\bar{y}_k$ is the
  leftmost but one (second from the left) of the real negative zeros of
  $\theta (\bar{q}_k,.)$).
\vspace{1mm}

(2) One has $-1<\bar{q}_{k+1}<\bar{q}_k<0$.
\vspace{1mm}

(3) For $q\in (\bar{q}_{k+1},\bar{q}_k)$,
  the function $\theta (q,.)$ has exactly
$k$ complex conjugate pairs of zeros counted with multiplicity. 
\end{tm}

The theorem is proved in Section~\ref{secprtmimprove}. In its proof we use
Proposition~\ref{prop2.2}.
We use also Proposition~\ref{propdoublezeros} which is of independent interest.

\subsection{Separating lines}

\begin{defi}\label{defiseparl}
  {\rm (1) For $q\in (\tilde{q}_1,1)$ and $a>0$, the line
    $\mathcal{S}_a:=\{ {\rm Re}\, x=-a\}$ in the plane of the variable $x$ is
    a {\em separating line} if all real zeros of $\theta (q,.)$
    are to its left and all complex conjugate pairs are to its right.
    \vspace{1mm}
    
    (2) For $q\in (-1,\bar{q}_1)$ and $a>0$,
    \vspace{1mm}
    
    (i) the line
    $\mathcal{L}_{a}:=\{ {\rm Re}\, x=-a\}$ is
    a {\em left separating line} if all negative real zeros of $\theta (q,.)$
    are to its left while all complex conjugate pairs
    and all positive real zeros are to its right;
    \vspace{1mm}
    
  (ii) the line
    $\mathcal{R}_a:=\{ {\rm Re}\, x=a\}$ is
    a {\em right separating line} if all negative real zeros of
    $\theta (q,.)$, its smallest positive real zero 
    and all complex conjugate pairs are to its left while all positive
    real zeros except the smallest one are to its right.}
\end{defi}

\begin{tm}\label{tmseparl}
  (1) For every $q\in (\tilde{q}_1,1)$, there exists a separating line with
  $a\geq 5$.
  \vspace{1mm}

  (2) For every $q\in (-1,\bar{q}_1)$, there exists a left separating line
  with $a\geq 2.4$.

  (3) For every $q\in (-1,\bar{q}_2)$, there exists a right separating line
  with $a\geq 3.2$.
\end{tm}

The theorem is proved in Section~\ref{secprtmseparl}. For the inequalities $a\geq 2.4$ and $a\geq 3.2$ see parts (1) and (3) of Proposition~\ref{prop2.2}.

\begin{rems}
  {\rm (1) One can choose the constants $a$ defining the separating lines
    (resp. the left
    or right separating lines) as continuous functions on each of the
    intervals $(\tilde{q}_j,\tilde{q}_{j+1})$ (resp. $(\bar{q}_{2j+1},\bar{q}_{2j-1})$ or $(\bar{q}_{2j+2},\bar{q}_{2j})$).
    They can be continuous on $(0,1)$ or $(-1,0)$ only if  
    one admits double zeros of $\theta$ to belong to the separating (resp. left/right separating) lines.
    \vspace{1mm}

    (2) Numerical computation with the truncation $\theta_{140}$
    of $\theta$ up to the
    140th term and with $q=-0.96$ shows that $\theta_{140}(-0.96,.)$
    has a pair of conjugate zeros $0.8246197382\ldots \pm 1.226652727\ldots$
    and a real zero in $[1,1+10^{-9}]$. One can conjecture that as $q$
    decreases in $(-1,0)$ and tends to $-1^+$, for every conjugate pair of
    zeros $a\pm ib$ of $\theta$, $a>0$, $a=a(q)$, $b=b(q)$,
    there exists $q_{\bullet}\in (-1,0)$ such that one has $a>x_1$,
    $a=x_1$ and $a<x_1$ for $q>q_{\bullet}$, $q=q_{\bullet}$ and $q<q_{\bullet}$
    respectively, where $x_1(q)$ is the first positive zero of $\theta (q,.)$.
    This is the only situation in which three zeros of $\theta$
    have the same real part.}
\end{rems}

\subsection{The signs of $\partial^2\theta /\partial x^2$
  and $\partial \theta /\partial q$}

The propositions and lemma from this subsection are proved in
Section~\ref{secprpropsecond}. 

\begin{prop}\label{propsecond}
  For $q\in (0,1)$ and $x\geq -q^{-3/2}$,
  the function $\partial^2\theta /\partial x^2$
  is positive (so it has no zeros). Hence the function
  $\partial \theta /\partial q$ is negative for $x\in [-q^{-1/2},0)$ and
    positive for $x\in (0,\infty )$.
    \end{prop}

From Proposition~\ref{propsecond} we deduce an interesting corollary 
about the functions $\varphi_k(q):=\theta (q,-q^{k-1})$ which have often been 
used in the study of the analytic properties of~$\theta$:

\begin{lm}\label{lmkgeq1}
  For $q\in (0,1)$ and $k=1/2$ or $k\geq 1$,
  the function $\varphi_k$ is
  decreasing.
\end{lm}

We set $K^{\dagger}:=\cup_{j=1}^{\infty}(2j-1,2j)\subset \mathbb{R}$.
The proposition that follows is a
more general statement than \cite[Proposition~3]{KoSe}. Lemma~\ref{lmkgeq1}
is involved in its proof.

\begin{prop}\label{propincr}
  For $q\in (0,1)$ and $a\in K^{\dagger}$, the function $\theta (q,-q^{-a})$
  is strictly increasing from $-\infty$ to $1/2$.
  For each $a\in K^{\dagger}$, there exists a unique
  point $(q_a,-q_a^{-a})$, $q_a\in (0,1)$, such that $\theta (q_a,-q_a^{-a})=0$. 
  For $a>0$, $a\not\in K^{\dagger}$, there
  exists no such point $(q_a,-q_a^{-a})$.
\end{prop}

At the end of Section~\ref{secprpropsecond} we prove

\begin{prop}\label{prop2kincrease}
  In case A) each even zero $\xi_{2k}$ of $\theta$ is an increasing function in  
  $q\in (0,\tilde{q}_k]$.
  \end{prop}

\section{The method of proof}

The partial theta function owes its name to its resemblance with the
{\em Jacobi theta function}
$\Theta (q,x):=\sum_{j=-\infty}^{\infty}q^{j^2}x^j$, because
$\theta (q^2,x/q)=\sum_{j=0}^{\infty}q^{j^2}x^j$. ``Partial'' means that
summation in the case of $\theta$ is performed only from $0$ to $\infty$,
not from $-\infty$ to $\infty$. 

In the proofs we use, except the equality (\ref{equqx}),
the {\em Jacobi triple product}

$$\Theta (q,x^2)=\prod_{m=1}^{\infty}(1-q^{2m})(1+x^2q^{2m-1})(1+x^{-2}q^{2m-1})~.$$
In the text this formula is applied mainly to the function

$$\Theta^*(q,x)=\Theta (\sqrt{q},\sqrt{q}x)=\sum_{j=-\infty}^{\infty}q^{j(j+1)/2}x^j$$
in which case it yields

\begin{equation}\label{equT}
  \begin{array}{ccl}
    \Theta^*(q,x)&=&\prod_{m=1}^{\infty}(1-q^m)(1+xq^m)(1+q^{m-1}/x)\\ \\
    &=&(1+1/x)\prod_{m=1}^{\infty}(1-q^m)(1+xq^m)(1+q^m/x)~.\end{array}
\end{equation}
Setting $G:=\sum_{j=-\infty}^{-1}q^{j(j+1)/2}x^j=1/x+q/x^2+q^3/x^3+q^6/x^4+\cdots$,
one can represent
$\theta$ in the form

\begin{equation}\label{equthetaG}
  \theta =\Theta^*-G~.
  \end{equation}
For large values of $|x|$ one obtains majorations of $|G|$. Formula (\ref{equT})
allows to study the variation of $|\Theta^*|$ as a function in~$x$.
In the proofs we often make use of computer algebra.

\section{Proofs of Propositions~\protect\ref{prop6} and~\protect\ref{prop2.2}
  \protect\label{secprprop2.2}}

\begin{proof}[Proof of part (1) of Proposition~\ref{prop6}]
  We prove that $\theta (q,-6)>0$ for $q\in (0,1)$ from which follows
  $\theta (q,x)>0$ for $q\in (0,1)$, $x\in [-6,0]$. Indeed, $\theta >0$
  for $x\in [-1/q,0]$ (see \cite[Proposition~7]{KoBSM1}), hence
  $\theta (q,x)>0$ for $x\in [-6,0]$ and $q\in (0,1/6]$. Real zeros
  are not born, but can only disappear as $q$ increases in $(0,1)$, so
  $\theta (q,x)>0$ for $(q,x)\in (0,1)\times [-6,0]$.

  For $q\in (0,0.95]$,
  we use the rapid convergence of the series of $\theta$. Its truncation
  to the $100$th term is $>0.0073$ for $x=-6$. The modulus of the $101$st term is
  $7.03\ldots \times 10^{-37}$ and the next terms decrease faster than a
  geometric progression with ratio $0.00321$. Therefore $\theta (q,-6)>0.007$
  for $q\in (0,0.95]$. 

  Suppose that $q\in [0.95,1)$ and $x=-6$. We use the representation
    (\ref{equthetaG}).
    For the Leibniz series (in $q$) $G|_{x=-6}$, we obtain the estimation

    $$-0.1666\ldots =-1/6<G|_{x=-6}<-1/6+q/36<-1/6+1/36=-0.1388\ldots ~.$$
    Now we
    consider the term $\Theta^*$, see (\ref{equT}).
The product $\prod_{m=1}^{\infty}(1-q^{m-1}/6)$ is with positive factors.
    It is majorized by
$\prod_{m=1}^{\infty}(1-0.95^{m-1}/6)<0.0312$.  

    Denote by $k$ the smallest natural number for which $6q^k\leq 1$. As
    $6\cdot 0.95^{34}=1.04\ldots$ and $6\cdot 0.95^{35}=0.996\ldots$, one has
    $k\geq 35$ for $q\in [0.95,1)$. Hence
    $6q^{k-1}>1$, $6q^k>q$, 
    
    $$|1-6q^{k-\nu}|=6q^{k-\nu}-1<6q^{k-\nu}~,~\nu =1,~2,~\ldots ,~k-1~~~\,
    {\rm and}~~~\, 1-6q^{k+j}<1-q^{j+1}~,~j=0,~1,~\ldots ~.$$
    Thus

    $$\begin{array}{rcl}|\prod_{m=1}^{\infty}(1-6q^m)|&=&
      |\prod_{m=1}^{k-1}(1-6q^m)|\cdot \prod_{m=k}^{\infty}(1-6q^{m})\\ \\
      &\leq&|\prod_{m=1}^{k-1}(1-6q^m)|\cdot \prod_{m=k}^{\infty}(1-q^{m-k+1})\\ \\
      &\leq&6^{k-1}q^{k(k-1)/2}\cdot \prod_{m=1}^{\infty}(1-q^{m})~,~~~\,
      {\rm so}\\ \\ \prod_{m=1}^{\infty}(1-q^m)\cdot |\prod_{m=1}^{\infty}(1-6q^m)|
      &\leq&6^{k-1}q^{k(k-1)/2}\cdot (\prod_{m=1}^{\infty}(1-q^m))^2~.\end{array}$$
    One has $q^{k(k-1)/2}=q\times q^2\times \cdots \times q^{k-1}$. The function
    $q(1-q)^2$, $q\in (0,1)$, takes its maximal value for $q=1/3$ which equals
    $4/27$. Hence

    $$0<q^m(1-q^m)^2\leq 4/27~,~m=1,~\ldots ,~k-1~~~\, {\rm and}$$
    $$\begin{array}{ccl}
      6^{k-1}q^{k(k-1)/2}\cdot (\prod_{m=1}^{\infty}(1-q^m))^2&<&
      6^{k-1}q^{k(k-1)/2}\cdot (\prod_{m=1}^{k-1}(1-q^m))^2\\ \\ &=&
      \prod_{m=1}^{k-1}(6q^m(1-q^m)^2)\leq (6\cdot 4/27)^{k-1}
      \leq (8/9)^{34}=0.018\ldots ~.\end{array}$$
    Finally, $|\Theta^*|<0.0312\cdot 0.018=0.0005616$ and
    $\theta >0.1388-0.0005616=0.1382384>0$.
  \end{proof}

\begin{proof}[Proof of part (2) of Proposition~\ref{prop6}]

  The border $\partial \tilde{Q}$ of the rectangle $\tilde{Q}$ consists of the
  segments

  $$\begin{array}{lcl}
    J_1:=\{ q=0.4,~x\in [-10.5,~0]\}~,&&J_2:=\{ x=-10.5,~q\in [0.4,~1]\}~,\\ \\ 
  J_3:=\{ x=0,~q\in [0.4,~1]\}&{\rm and}&
  J_4:=\{ q=1,x\in [-10.5,~0]\}~.\end{array}$$
  For $(q,x)\in J_3$ (resp. $(q,x)\in J_4$), one has $\theta \equiv 1$ (resp.
  $\theta =1/(1-x)$, see Remark~\ref{remKatsnelson}), so $\theta >0.0049$.
  It suffices to prove that $\theta (q,x)>0.0049$ for
  $(q,x)\in J_1\cup J_2$. Indeed, the zeros of $\theta$
  depend continuously on $q$ and as $q$ increases in $(0,1)$, no real zeros are
  born, but can only be lost. If no zero crosses the set
  $\partial \tilde{Q}$, then there are no zeros of $\theta$ inside
  $\tilde{Q}$.

      For $(q,x)\in J_1$, one finds numerically that
      $\theta (q,x)\geq 1.4>0$. Indeed, this is the case of its
      truncation up to the 300th term. For $x=-10.5$, the latter equals
      $2.3\ldots \times 10^{-17661}$ and the moduli of the subsequent
      terms decrease
      at least as fast as a geometric progression with ratio
      $1.7\ldots \times 10^{-119}$.

      Similarly for $(q,x)\in [0.4,0.87]\times \{ -10.5\}$,
      one finds that the same
      truncation takes only values $\geq 0.005$. For $q=0.87$, the 300th term
      equals $4.4\ldots \times 10^{-2425}$ and the moduli of the terms
      decrease faster than 
      a progression with ratio $6.5\times 10^{-18}$.
So now we
      concentrate on the case $x=-10.5$, $q\in [0.87,1]$.
        Consider for $q\in [0,1]$ the function

$$U(q):=(1-q)\cdot (1-q/10.5)\cdot (1-10.5\cdot q)$$
        and the product $\tilde{U}:=\prod_{m=1}^{\infty}|U(q^m)|$.
        (One needs to study $U$ on the whole interval $[0,1]$, because of the
        presence of powers $q^m$ in $\tilde{U}$.) 
        The
minimal value of $U$ (attained for $q=0.5355870925\ldots$) is
$-2.037759887\ldots >-2.04$. We list the solutions in $(0,1)$
to certain equations:

$$\begin{array}{lll}
  U(q)=1/1.5&b=0.02962269556\ldots&\\ \\
  U(q)=1/2.04&t=0.04608175109\ldots&\\ \\
  U(q)=-1/2.04&s_1=0.1510312684\ldots&s_2=0.9392525427\ldots\\ \\
  U(q)=-1/1.5&w_1=0.1733320522\ldots&w_2=0.9151706313\ldots\\ \\ 
  U(q)=-1&m_1=0.2199411286\ldots&m_2=0.8652006788\ldots\\ \\ 
  U(q)=-1.5&r_1=0.3078774331\ldots&r_2=0.7722315515\ldots
  \end{array}$$
\begin{nota}
  {\rm For each $q\in (0,1]$ fixed, we denote by
    $\sharp ([\alpha ,\beta ])$ the quantity of numbers of the
    form $q^m$ contained in the interval $[\alpha ,\beta ]$.}
\end{nota}
Suppose that $q\in [0.87,1]$. One has

  $$r_2/r_1=2.508\ldots <2.51<2.851393456\ldots =0.87\cdot (s_1/t)~,
  ~~~\, {\rm where}~~~\, s_1/t=3.277463743\ldots ~.$$
This means that $\sharp ([r_1,r_2])\leq \sharp ([t,s_1])~(A)$. Next, 
  
    $$r_1/m_1=1.399817465\ldots <t/b=1.555623154\ldots ~,~~~\, {\rm so}~~~\, 
  \sharp ([m_1,r_1])\leq \sharp ([b,t])-1~(B)$$
  and in the same way

  $$m_2/r_2=1.120390221\ldots <1.147656734\ldots =w_1/s_1~,~~~\,
  {\rm so}~~~\, \sharp ([r_2,m_2])\leq \sharp ([s_1,w_1])-1~(C)~.$$
  We majorize now the quantities $|U(q^m)|$ (see the product $\tilde{U}$). 
  If $q^m\in I$, then $|U(q^m)|\leq \delta$,
  where the correspondence between $I$ and $\delta$
  is given by the rule:

  $$\begin{array}{ll}
    {\rm For}~q^m\in [b,t) \cup (s_1,w_1]~,~~~\, \delta =1/(1.5)~;&
    {\rm for}~q^m\in [t,s_1]~,~~~\, \delta =1/(2.04)~;\\ \\ 
    {\rm for}~q^m\in [m_1,r_1) \cup (r_2,m_2]~,~~~\, \delta =1.5~;&
    {\rm for}~q^m\in [r_1,r_2]~,~~~\, \delta =2.04~.
  \end{array}$$
  This rule and the inequalities $(A)$, $(B)$ and $(C)$ imply the inequalities

  \begin{equation}\label{equ3ineq}
    \begin{array}{lll}
      \prod_{q^m\in [r_1,r_2]}|U(q^m)|\prod_{q^m\in [t,s_1]}|U(q^m)|\leq 1~,&&
      \prod_{q^m\in [m_1,r_1)}|U(q^m)|\prod_{q^m\in [b,t)}|U(q^m)|\leq 1.5~,\\ \\
           {\rm and}&&
           \prod_{q^m\in (r_2,m_2]}|U(q^m)|\prod_{q^m\in (s_1,w_1]}|U(q^m)|\leq 1.5~.
    \end{array}
    \end{equation}
  The factors $|U(q^m)|$ in $\tilde{U}$ with $q^m$ outside these $6$
  intervals are of modulus $<1$. Hence $\tilde{U}\leq 1.5^2=2.25$.

  This estimation, however, is not sufficient and we give a better one
  as follows. Set $\rho :=0.87$ and $\eta :=1/10.5$.
  Then for $q\in [\rho ,1)$ and for
    $\varepsilon >0$ small enough,
    each of the following $7$ non-intersecting intervals
    contains at least one number $q^m$; 
    all intervals are subsets of the interval $(t,s_1)$ due to
    $\eta \rho^{9/2}=0.050\ldots >t$ and $\eta \rho^{-5/2}=0.134\ldots <s_1$:

    $$\begin{array}{ll}
      I_1:=[\eta \rho^{9/2}-4\varepsilon ,
      \eta \rho^{7/2}-3\varepsilon )~,&
  I_2:=[\eta \rho^{7/2}-3\varepsilon ,\eta \rho^{5/2}-2\varepsilon )~,\\ \\ 
        I_3:=[\eta \rho^{5/2}-2\varepsilon ,\eta \rho^{3/2}-\varepsilon )~,&
      I_4:=[\eta \rho^{3/2}-\varepsilon ,\eta \rho^{1/2})~,\\ \\      
I_5:=[\eta \rho^{1/2},\eta \rho^{-1/2}]~,& 
I_6:=(\eta \rho^{-1/2},\eta \rho^{-3/2}+\varepsilon ]~,\\ \\ {\rm and}&
      I_7:=(\eta \rho^{-3/2}+\varepsilon ,\eta \rho^{-5/2}+2\varepsilon ]~.
      \end{array}$$
    The quantity $|U(q^m)|$ with $q^m\in I_5$ is majorized by

    $$\max (|U(\eta \rho^{1/2})|,|U(\eta \rho^{-1/2})|)<0.1308043268=:u_5$$
    while $|U(q^m)|$ with $q^m\in I_j$, $1\leq j\leq 7$, $j\neq 5$,  
    is majorized by a number arbitrarily
    close to

    $$\begin{array}{ll}
      |U(\eta \rho^{9/2})|<0.4397976129=:u_1~,&
      |U(\eta \rho^{7/2})|<0.361198525=:u_2~,\\ \\
      |U(\eta \rho^{5/2})|<0.2724861417=:u_3~,&
      |U(\eta \rho^{3/2})|<0.1726682682=:u_4~,\\ \\
      |U(\eta \rho^{-3/2})|<0.202756323=:u_6~,&
      |U(\eta \rho^{-5/2})|<0.355643912=:u_7~.
    \end{array}$$
    We remind that for the moment we know only that $\tilde{U}\leq 2.25$.
    The $7$ quantities $|U(q^m)|$ majorized by $u_j$ were initially
    majorized by $1/2.04$. Hence

    $$\tilde{U}\leq 2.25\cdot (\prod_{j=1}^7u_j)\cdot 2.04^7=
    0.01036522792\ldots <0.0104~.$$
%

          At the same time for $x=-1/10.5$ and $q\in [0.87,1)$, one minorizes
            $|G|$ by $1/10.5-1/10.5^2=0.086\ldots$, because $G$ is a
            Leibniz series. Hence for $x=-1/10.5$ and $q\in [0.87,1)$, one has
              $\theta (q,x)>0.086-0.0104=0.0756>0.0049$.

\end{proof}

\begin{proof}[Proof of Proposition~\ref{prop2.2}]

  Part (1). For $q\in (-1,0)\cup (0,1)$, equality (\ref{equdecomp})
  holds true.
  For $x<0$ and $q\in (-1,0)$, the quantities $x^2/q$ and $qx^2$ are negative
  while $qx>0$ and $q^4\in (0,1)$.
  For $q\in (0,1)$, any real zero of $\theta (q,.)$ is $<-6$
  (see Proposition~\ref{prop6}), and one has 
  $\theta (q,0)=1>0$, so it is also true that 
  $\theta (q,x)>0$ for $x\in [-6,0]$. Hence for $q\in (-1,0)$ and $x<0$,
  one has $\theta_1>0$ and $qx\theta_2>0$
  when $x^2/q\in [-6,0]$ and $qx^2\in [-6,0]$. The latter two inequalities
  hold true for $x\in [-2.4,0]$ and $q\in (-1,-0.97]$. Thus it remains to prove
      part (1) of the proposition for $q\in (-0.97,0)$.

      The function $\theta_{\diamond}:=\sum_{j=0}^{100}(-2.4)^jq^{j(j+1)/2}$
      is a polynomial in $q$. One finds numerically that its minimal value for
      $q\in [-0.97,0]$ is $>0.2$. The sum
      $$S_*:=\sum_{j=101}^{\infty}|(-2.4)^j(-0.97)^{j(j+1)/2}|~,$$
      which majorizes $|\sum_{j=101}^{\infty}(-2.4)^jq^{j(j+1)/2}|$, is $<10^{-29}$.
      Indeed, its
      first term equals $1.8\ldots \times 10^{-30}$ and the terms decrease
      faster than a geometric progression with ratio
      $2.4\cdot 0.97^{102}=0.107\ldots$. Hence 
      $\theta (q,-2.4)>0$ for $q\in [-0.97,0)$. This proves part~(1).

        Part (2). We remind first that for $q\in (-1,0)$, one has
        $\theta (q,-1/q)<0$,
        see~\cite[Proposition~4.5]{KoPRSE2}, and
        $\theta (q,1)>0$; the latter inequality follows from
        $\theta (q,1)=\sum_{j=0}^{\infty}q^{j(j+1)/2}=
        \frac{1-q^2}{1-q}\cdot \frac{1-q^4}{1-q^3}\cdot \frac{1-q^6}{1-q^5}
        \ldots$, 
        see Problem 55 in Part I, Chapter 1 of~\cite{PoSz}. From these two
        inequalities one deduces that the leftmost positive zero of
        $\theta (q,.)$ belongs to the interval $(1,-1/q)$
        and the next one is $>-1/q$.
        For $q\in (-5/12,0)$ (with $5/12=0.416666\ldots$),
        one has $2.4<-1/q$, so the second positive zero
        is $>2.4$. For $q\in [-0.97,-5/12]$, one considers the polynomial
        (in $q$)
        $\theta_{\ddagger}:=\sum_{j=0}^{100}2.4^jq^{j(j+1)/2}$. One finds numerically that
        its maximal value is $<-0.1$. On the other hand the sum $S_*$ 
        is $<10^{-29}$, so for
        $q\in [-0.97,0)$, one has $\theta (q,2.4)<0$
          and the second positive zero is $>2.4$.

          For $q\in (-1,-0.97)$, we use the decomposition (\ref{equdecomp}).
          The functions $\theta_1$ and $\theta_2$ are even functions in~$x$
          for every fixed $q\in (-1,0)$. The
          smallest positive zeros $\zeta_1$ and $\zeta_2$ of
          $\theta_1$ and $\theta_2$ are $>2.4$ and $\zeta_2>\zeta_1$, so

          $$\theta_1|_{x=\zeta_1}=0>(qx\theta_2)|_{x=\zeta_1}~~~\,
          {\rm hence}~~~\, \theta (q,\zeta_1)<0~.$$
          This means that the second positive zero of $\theta (q,.)$
          is $>\zeta_1$ hence $>2.4$ for $q\in (-1,-0.97)$,
          so for $q\in (-1,0)$ as well.

          Part (3). We remind that $\theta (q,0)\equiv 1$ and
          $\theta (-1,x)=(1-x)/(1+x^2)$, see Remark~\ref{remKatsnelson}.
          Suppose first that
          $(q,x)\in \tilde{Q}_-$. One finds numerically
          that for $q=-0.75$ and
          $x\in [-4,0]$, and for $x=3.1$ and $q\in [-0.97,-0.75]$,
          the function $\theta$ takes only values $>0.2$. Suppose that
          $x=3.1$ and $q\in [-1,-0.97]$. Hence in the equality (\ref{equdecomp})
          both terms $\theta_1$ and $qx\theta_2$ are positive. Indeed, $qx>0$,
          $q^4>0.4$ and $-10.5<x^2/q<qx^2<0$, so one can apply
          part (2) of Proposition~\ref{prop6}. As $\theta_1>0.0049$,
          one concludes that $\theta >0.0049$.

          Consider now the case $(q,x)\in \tilde{Q}_+$. Numerical check shows
          that for $x=3.2$ and $q\in [-0.94,-0.8]$, $\theta <-0.08$; and for
          $q=-0.8$ and $x\in [0,4]$, $\theta$ has a single positive zero
          (which is close to $1$).

          Suppose that $x=3.2$ and $q\in [-1,-0.94]$. At the first positive
          zero of $\theta_2$ one has $\theta <0$, so the second positive zero of
          $\theta$ is larger than $x_*$, where $x_*^2q=-10.5$,
          see part~(2) of Proposition~\ref{prop6}. Hence
          $x_*>10.5^{1/2}=3.240370349\ldots >3.2$ and for $(q,x)\in \tilde{Q}_+$,
          the only real zero of $\theta$ is its smallest positive zero. 
          
          For $q=-0.78$, the first three positive zeros of $\theta$ equal
          $1.02\ldots$, $2.75\ldots$ and $3.16\ldots$. They are $<3.2$.

\end{proof}

\section{Proof of Theorem~\protect\ref{tmseparl}\protect\label{secprtmseparl}}

\subsection{Proof of part (1) of Theorem~\protect\ref{tmseparl}}

We use the representation (\ref{equthetaG}). 
We consider separately $|\Theta^*(q,x)|$ and $|G(q,x)|$.

\begin{lm}\label{lmGb}
  Suppose that $a\geq 5$, $b\geq 0$
  and $q\in (0,1)$. Set $x:=-a+bi$. Then for each $(a,q)$ fixed,
  the quantity $|G|$ is majorized by a decreasing
  function in $b$ which coincides with $|G|$ for $b=0$.
\end{lm}

The lemmas of this subsection are proved at its end.

\begin{lm}\label{lmFU}
  Set $x:=-a+bi$, $a\geq 5$, $b\geq 0$,

  $$F_k:=|(1+q^kx)(1+q^k/x)|^2~,~~~\, k\in \mathbb{N}^*~~~\,
  {\rm and}~~~\, U:=|(1+1/x)(1+qx)(1+q/x)|^2~.$$
  Hence $|\Theta^*|^2=(\prod_{k=1}^{\infty}(1-q^k))^2\cdot
  U\cdot \prod_{k=1}^{\infty}F_k$. 
  \vspace{1mm}
  
  (1) For $q\in (0,1)$ and $k\in \mathbb{N}^*$, one has
  $\partial F_k/\partial b\geq 0$
  with equality only for $b=0$.
  \vspace{1mm}
  
  (2) For $q\in [0.3,1)\supset (\tilde{q}_1,1)$, one has
    $\partial U/\partial b\geq 0$
    with equality only for $b=0$.
\end{lm}

In part (2) of the lemma we limit the proof to the case $q\geq 0.3$,
because there are no complex zeros for $q\leq 0.3$.
The lemma implies $\partial |\Theta^*|/\partial b>0$. 
For $\xi_j$ a real zero of $\theta (q,.)$, one has $-\xi_j>5$, see part (1)
of Proposition~\ref{prop6}.
For $x=\xi_j$, it is true that $|\Theta^*|=|G|$,
because $(\Theta^*-G)(q,\xi_j)=0$.
Lemmas~\ref{lmGb} and \ref{lmFU} imply that for $x=\xi_j\pm bi$, $b>0$,
one has $|\Theta^*(q,x)|>|G(q,x)|$, so $\theta (q,x)\neq 0$.

Suppose that $q\in (\tilde{q}_1,\tilde{q}_2]$. Then there is just
one complex conjugate pair of $\theta (q,.)$ and the condition
$\theta (q,\xi_3\pm bi)\neq 0$ means that for
$\varepsilon >0$ sufficiently small, the line
$\mathcal{S}_{-\xi_3-\varepsilon}$ is a separating line.

For $q\in (\tilde{q}_2,\tilde{q}_3]$ and $\varepsilon >0$ sufficiently small,
  the line
  $\mathcal{S}_{-\xi_5-\varepsilon}$ is a separating line.
  Indeed, for $q=\tilde{q}_2$, the complex conjugate pair of $\theta$
  is to the right of the line $\mathcal{S}_{-\xi_3-\varepsilon}$ hence to the right
  of $\mathcal{S}_{-\xi_5-\varepsilon}$ as well. For $\eta >0$ sufficiently small,
  the complex conjugate pair born from the double zero for $q=\tilde{q}_2$
  remains to the right of the line $\mathcal{S}_{-\xi_5-\varepsilon}$ for
  $q=\tilde{q}_2+\eta$. Hence it remains to its right for
  $q\in (\tilde{q}_2,\tilde{q}_3]$. 

  Continuing like this one finds that for $q\in (\tilde{q}_k,\tilde{q}_{k+1}]$,
    $\mathcal{S}_{-\xi_{2k+1} -\varepsilon}$ is a separating line.

    \begin{rem}
      {\rm The above reasoning shows also that for $\varepsilon >0$
        sufficiently small and for $q\in (\tilde{q}_k,\tilde{q}_{k+1}]$,
      to the right of the line $\mathcal{S}_{-\xi_{2s+1} -\varepsilon}$, $s>k$,
      remain all complex conjugate pairs and the real zeros
      $\xi_{2k+1}$, $\xi_{2k+2}$, $\ldots$, $\xi_{2s}$, while all other real
      zeros remain to its left.}
        \end{rem}

\begin{proof}[Proof of Lemma~\ref{lmGb}] 
Consider the sum of
  two consecutive terms of~$G$:

  $$t_k:=q^{(2k-1)(k-1)}x^{-2k+1}+q^{(2k-1)k}x^{-2k}=
  q^{(2k-1)(k-1)}x^{-2k}(x+q^{2k-1})~,~~~\, k=1,~2,\ldots ~.$$
Thus $G=\sum_{k=1}^{\infty}t_k$.
  Set $q_*:=q^{2k-1}$ and $q_{\dagger}:=q^{(2k-1)(2k-2)}$.
  One checks directly that

  $$|x|^2=a^2+b^2~~~\, {\rm and}~~~\, |x+q_*|^2=
  (a-q_*)^2+b^2~,~~~\, {\rm so}~~~\,
  |t_k|^2=q_{\dagger}((a-q_*)^2+b^2)/(a^2+b^2)^{2k}$$
and hence  

$$\partial (|t_k|^2)/\partial b=
2q_{\dagger}b((a^2+b^2)(1-2k)+4kq_*a-2kq_*^2)/(a^2+b^2)^{k+1}~.$$
For $k\in \mathbb{N}^*$, $b\geq 0$ and $a\geq 5$, one has
$\partial (|t_k|^2)/\partial b\leq 0$ with equality only for $b=0$. Indeed,

$$-(2k-1)(a^2+b^2)+4kq_*a-2kq_*^2\leq -(2k-1)a^2+4kq_*a-2kq_*^2=:
\mathcal{H}(a)~,$$
where
$$\mathcal{H}(5)=-25(2k-1)+20kq_*-2kq_*^2<5(-10k+5+4kq_*)<5(-6k+5)<0$$
and for $a\geq 5$, 
$$\mathcal{H}'(a)=-2a(2k-1)+4kq_*\leq -10(2k-1)+4kq_*
\leq -20k+10+4k=-16k+10<0~.$$

For $b=0$, all quantities $t_k$ are negative real numbers,
i.~e. they have
  the same argument. Therefore for $b=0$, the modulus of their sum equals
  the sum of their moduli. For $b>0$,
  the modulus of each of them is smaller than its modulus for $b=0$.
  Thus $|G|\leq \sum_{k=1}^{\infty}|t_k|$, where the right-hand side
  is a decreasing
  function in $b$ for $b\geq 0$ and there is equality only for $b=0$.
\end{proof}

\begin{proof}[Proof of Lemma~\ref{lmFU}]
  Part (1). One checks directly that

  $$\begin{array}{ccccc}
    \partial F_k/\partial b&=&2bq^kF_{k,1}/(a^2+b^2)^2~,&&{\rm where}\\ \\
    F_{k,1}&:=&(a^2+b^2)^2q^k-4a^2q^k+2aq^{2k}+2a-q^k&\geq&
    a^4q^k-4a^2q^k+2aq^{2k}+2a-q^k\\ \\
    &=&(a^2-4)a^2q^k+2aq^{2k}+(2a-q^k)&>&0\end{array}$$
    from which part (1) follows.

 Part (2). It is true that $\partial U/\partial b=2bH/(a^2+b^2)^3$, where
 $H:=K+Lb^2+3a^2q^2b^4+q^2b^6$, with 
    $$\begin{array}{cclccl}
      K&:=&q^2a^6+K_4a^4+K_3a^3+K_2a^2+K_1a+K_0~,&K_0&:=&-2q^2~,\\ \\
      K_1&:=&4q^3+4q^2+4q~,&K_2&:=&-q^4-8q^3-9q^2-8q-1~,\\ \\
      K_3&:=&2q^4+4q^3+16q^2+4q+2~,&K_4&:=&-4q^3-4q^2-4q\end{array}$$
      and

      $$\begin{array}{cclcccl}
        L&:=&3q^2a^4+L_2a^2+L_1a+L_0~,&&L_0&:=&-q^4-q^2-1~,\\ \\
        L_1&:=&2q^4+4q^3+4q+2~,&&L_2&:=&-4q^3-4q^2-4q~.\end{array}$$
      We show that $K>0$ and $L>0$ which implies part~(2).
      It is clear
      that for $a\geq 5$ and $q\in [0.3,1)$, one has $L_1a>|L_0|$ and

        $$3q^2a^2\geq 0.9\cdot 25\cdot q=23.4q~,~~~\, {\rm so}~~~\,
        3q^2a^4>12qa^2>(4q^2+4q+4)qa^2=|L_2|a^2$$
        hence $L>0$. Next, $K_1a>|K_0|$ and $K_3a^3>|K_2|a^2$. The function
        $f:=25q-(4q^2+4q+4)$ is increasing for $q\in [0.3,1)$ (one has
          $f'=25-8q-4>13>0$) and $f(0.3)=1.94$. Therefore

          $$q^2a^6+K_4a^4=a^4q(qa^2-4q^2-4q-4)
          \geq a^4qf(q)\geq 5^4\cdot 0.3\cdot 1.94>0~,$$
          so $K>0$, $H>0$ and $\partial U/\partial b\geq 0$ with equality only
          for $b=0$.

\end{proof}

\subsection{Proof of part (2) of Theorem~\protect\ref{tmseparl}}

We set $x:=-a+bi$ (where $a,b\in \mathbb{R}_+$) and
$\tau :=-q$, so $\tau \in (0,1)$.

\begin{lm}\label{lm2.2leftQ}
  For fixed $\tau$ and $a\geq 2.4$,
  the quantity $|\Theta^*|$ is an increasing function
  in~$b$.
\end{lm}
The lemmas of this subsection are proved at its end.

\begin{lm}\label{lm2.2leftG}
  For $a\geq 2.4$ and $q$ fixed, the quantity $|G|$
  is majorized by a positive-valued function which is decreasing in~$b$
  and which equals $|G|$ for $b=0$.
  \end{lm}

The proof of part (2) is deduced from the above two lemmas in the same way as
the proof of part (1) follows from Lemmas~\ref{lmGb} and~\ref{lmFU},
see the lines after Lemma~\ref{lmFU}.

\begin{proof}[Proof of Lemma~\ref{lm2.2leftQ}]  
  As $q=-\tau$ is fixed, one needs to consider only the factors
  (see (\ref{equT}))

  $$1+1/x~,~~~\, 1+q^mx=1+(-1)^m\tau^mx~~~\, {\rm and}~~~\,
  1+q^m/x=1+(-1)^m\tau^m/x~,~~~\, m\in \mathbb{N}^*~.$$
  We study first the product of three
factors $P:=(1+1/x)(1-\tau x)(1-\tau /x)$. Using computer algebra one finds that

\begin{equation}\label{equP}
  \partial (|P|^2)/\partial b=2bP_1/(a^2+b^2)^3~,~~~\,
P_1=A\tau ^4+B\tau ^3+C\tau ^2+B\tau +A~,\end{equation}
where
$$\begin{array}{ccl}
  A&:=&2a^3+2ab^2-a^2-b^2=a^2(2a-1)+b^2(2a-1)>0~,\\ \\
  B&:=&4a^4+4a^2b^2-4a^3-4ab^2+8a^2-4a=4(a^3+ab^2)(a-1)+4a(2a-1)>0~~~\,
  {\rm and}\\ \\
  C&:=&a^6+3a^4b^2+3a^2b^4+b^6-4a^4-4a^2b^2+16a^3-9a^2-b^2+4a-2\\ \\
  &=&a^4(a^2-4)+a^2b^2(2a^2-4)+a^2(16a-9)+(a^4-1)b^2+(4a-2)>0~,\end{array}$$
so $\partial (|P|^2)/\partial b>0$ and for $b\geq 0$, $|P|$ is minimal
for and only for $b=0$. (The inequality $C>0$ follows from
$a^6-4a^4+16a^3-9a^2>0$, true for $a\geq 1.5$.)  

The remaining factors will be considered in products by four. Namely, set  
$t:=q^{2k}$ and $-T:=q^{2k+1}$ hence $0<T<t<1$.
Consider the product 

$$R:=(1+tx)(1+t/x)(1-Tx)(1-T/x)~.$$
Set $R_b:=\partial (|R|^2)/\partial b$. By means of computer algebra
one finds that $R_b=2bH/(a^2+b^2)^3$, where

$$\begin{array}{ccl}H&=&2T^2t^2b^8+H_6b^6+H_4b^4+H_2b^2+H_0~,~~~\,
  H_i\in \mathbb{R}[a,t,T]~,~~~\, {\rm with}\\ \\ 
  H_6&=&8T^2t^2a^2+2\beta a+\gamma ~,~\beta :=T^3t^2-T^2t^3-T^2t+Tt^2~,~
  \gamma :=T^4t^2+T^2t^4+T^2+t^2~.\end{array}$$
The discriminant

$$\begin{array}{ccl}\Delta &:=&\beta ^2-8T^2t^2\gamma \\ \\
  &=&-T^2t^2(7T^4t^2+2T^3t^3+7T^2t^4+2T^3t+2Tt^3+7T^2+2Tt+(7t^2-4T^2t^2))
\end{array}$$
is negative, so $H_6>0$ for $a\in \mathbb{R}$. In the same way
$H_4=3a^2(4T^2t^2a^2+2\beta a+\gamma )$, where the discriminant 

$$\tilde{\Delta}:=-T^2t^2(3T^4t^2+2T^3t^3+3T^2t^4+2T^3t+2Tt^3+2Tt+
(3t^2+3T^2-4T^2t^2))$$
is negative, so $H_4>0$ for $a\in \mathbb{R}$.

Next, set $H_2:=H_2|_{t=T}+(t-T)H_{2*}$, where

$$\begin{array}{ccl}
  H_2|_{t=T}&=&2T^2((3T^4a^4-T^4)+4T^2a^6+8T^2a^2+(3a^4-1))~~~\,
  {\rm and}\\ \\
  H_{2*}&=&6T^5a^4+6T^4a^4t+8T^3a^6+3T^3a^4t^2+(8T^2a^6t-6T^2a^5t^2)+3T^2a^4t^3\\ \\
  &&+(4T^3a^2t^2-2T^3at^3)+(8T^2a^3t^2-4T^2a^2t^3)+(6Ta^5t-2T^5-2T^4t)\\ \\
  &&+(16T^3a^2-2T^3at-T^3t^2)+(8T^2a^2t-T^2t^3)+(8Ta^2t^2-2Tat^3)+2T^2a\\ \\ 
&&+2at^2+(2a-t-T)+[4Ta^2+3Ta^4-8Ta^3t+3a^4t-4a^2t]~.\end{array}$$
The sums of the terms put between brackets are easily shown
to be non-negative. The sum $S^*$ of terms between the
square brackets can be transformed as follows:

$$\begin{array}{ccl}
  S^*&=&(3a^4+4a^2)T+(3a^4-4a^2)t-8Ta^3t\geq (3a^4+4a^2)T+(3a^4-4a^2)T-8Ta^3\\ \\
  &=&(6a^4-8a^3)T>0~~~\, {\rm (true~for}~a\geq 2.4{\rm )}~.\end{array}$$
Hence $H_2\geq 0$. Similarly one sets $H_0:=H_0|_{t=T}+(t-T)H_{0*}$, where

$$\begin{array}{ccl}
  H_0|_{t=T}&=&2T^2(T^4a^6+(T^2a^8-8T^2a^4)+3T^4a^2+(a^6-T^2)+3a^2)>0~,\\ \\ 
  &&{\rm because~}2.4^4>8~{\rm ,~so~}a^8>8a^4~{\rm ,~and}\\ \\ 
  H_{0*}&=&
  (2T^3+2T^2t)a^8+(-2T^2t^2+2Tt)a^7+(2T^5+2T^4t+T^3t^2+T^2t^3+T+t)a^6\\ \\
  &&(8T^2t^2-8Tt)a^5+(4T^3t^2-4T^2t^3-16T^3-24T^2t+8Tt^2+4T-4t)a^4\\ \\
  &&+(-2T^3t^3-2T^3t+16T^2t^2-2Tt^3+2T^2-16Tt+2t^2+2)a^3\\ \\
  &&+(6T^5+6T^4t+7T^3t^2-T^2t^3-8T^2t+8Tt^2+7T-t)a^2\\ \\
  &&+(4T^2t^2-4Tt)a-2T^3-2T^2t~.
\end{array}$$
We observe that $\eta :=-T^2t^2+Tt\geq 0$. The sum of the
terms in $a^7$, $a^5$ and $a$ equals

$$\eta (2a^7-8a^5-4a)\geq 0~~~\, {\rm for}~~~\, a\geq 2.4~.$$
One gets also 

$$t(a^6-4a^4-a^2)\geq 0~~~\, {\rm for}~~~\, a\geq 2.4~.$$
For the rest of the coefficient of $a^2$ it is clear that 

$$6T^5+6T^4t+7T^3t^2+6T +(T-T^2t^3)+(-8T^2t+8Tt^2)\geq 6T~.$$
The coefficient of $a^3$ can be majorized by $-18Tt+2t^2+2$, because 

$$-2T^3t^3-2T^3t+16T^2t^2\geq 0~~~\, {\rm and}~~~\, -2Tt^3+2Tt\geq 0~.$$
Having in mind that $(2T^3+2T^2t)(a^8-8a^4-1)\geq 0$, one can write

$$\begin{array}{ccl}H_{0*}&\geq& (2T^5+2T^4t+T^3t^2+T^2t^3+T)a^6+6Ta^2\\ \\ 
 &&+(4T^3t^2+(4T-4T^2t^3)+(8Tt^2-8T^2t))a^4+(-18Tt+2t^2+2)a^3~.  
  \end{array}
$$
The coefficient of $a^4$ is $\geq 0$. It is evident that $-4Tt+2t^2+2\geq 0$ and that for $a\geq 2.4$, 

$$ Ta^6+6Ta^2-14Tta^3\geq a^3T(a^3+6/a-14)\geq 0~.$$
Thus $H_{0*}\geq 0$ and 
$H>0$.

\end{proof}

\begin{proof}[Proof of Lemma~\ref{lm2.2leftG}]
  
We set $\tau :=-q\in (0,1)$
and we consider the sum

$$U:=1/x-\tau /x^2-\tau^3/x^3+\tau^6/x^4+\tau^{10}/x^5-\tau^{15}/x^6-\tau^{21}/x^7+
\tau^{28}/x^8$$
of the first $8$ terms of $G$. For $b=0$, one has $U<0$. Indeed, set $y:=-x$. 
For $y\geq 2.4$, one obtains 

$$U=(-1/y+\tau^3/y^3)+(-\tau /y^2+\tau^6/y^4)+(-\tau^{10}/y^5+\tau^{21}/y^7)+
(-\tau^{15}/y^6+\tau^{28}/y^8)~,$$
with negative sum of the terms between the brackets. 
Using computer algebra we find that

$$\partial (|U|^2)/\partial b=
-2b(b^{14}+\sum_{j=0}^6U_{2j}b^{2j})/(a^2+b^2)^9~,~~~\, {\rm where}$$

$$\begin{array}{ccl}U_{12}&=&7a^2+4a\tau +2\tau^2+4\tau^3~,\\ \\
  U_{10}&=&21a^4+24a^3\tau +(12\tau ^3+12\tau ^2)a^2+(18\tau ^6-6\tau ^4)a+
  6\tau ^{10}+6\tau ^7+3\tau ^6~,\\ \\
  U_8&=&35a^6+60a^5\tau +30a^4\tau ^2+(58\tau ^6-30\tau ^4)a^3+(-34\tau ^{10}+
  14\tau ^7+15\tau ^6)a^2\\ \\
  &&+(40\tau ^{15}-24\tau ^{11}+8\tau ^9)a+8\tau ^{21}+8\tau ^{16}+8\tau ^{13}+
  4\tau ^{12}~.\end{array}$$
One has $U_{12}\geq 0$ (evident), $U_{10}\geq 0$
(follows from $12\tau ^2a^2\geq 6\tau ^4a$) and $U_8\geq 0$ (results from
$30a^4\tau ^2\geq 30\tau ^4a^3$ and
$60a^5\tau \geq 34\tau ^{10}a^2+24\tau ^{11}a$). The next
coefficient~is 

$$\begin{array}{ccl}U_6&=&35a^8+80a^7\tau +(-40\tau ^3+40\tau ^2)a^6+
  (52\tau ^6-60\tau ^4)a^5+
  (-116\tau ^{10}-4\tau ^7+30\tau ^6)a^4\\ \\
  &&+(-40\tau ^{15}-56\tau ^{11}+32\tau ^9)a^3+(-148\tau ^{21}-48\tau ^{16}+
  12\tau ^{13}+16\tau ^{12})a^2\\ \\
  &&+(70\tau ^{28}-50\tau ^{22}+30\tau ^{18}-10\tau ^{16})a+10\tau ^{29}+
  10\tau ^{24}+10\tau ^{21}+5\tau ^{20}~.
\end{array}$$
We explain how one can prove that $U_6>0$; a similar method will be applied
to $U_4$, $U_2$ and $U_0$. We majorize the coefficients of the powers of $a$
using the inequalities $0\leq \tau \leq 1$:

$$\begin{array}{ll}-40\tau ^3+40\tau ^2\geq 0~,&52\tau ^6-60\tau ^4\geq
  -60\tau ^4\geq -60~,\\ \\ 
  -116\tau ^{10}-4\tau ^7+30\tau ^6\geq -90\tau ^{10}\geq -90~,&
  -40\tau ^{15}-56\tau ^{11}+32\tau ^9\geq -64\tau ^{11}\\ &\geq -64~,\\ 
  -148\tau ^{21}-48\tau ^{16}+12\tau ^{13}+16\tau ^{12}\geq -168\tau ^{16}\geq
  -168~,&\\ \\ 
  70\tau ^{28}-50\tau ^{22}+30\tau ^{18}-10\tau ^{16}\geq 70\tau ^{28}
  -20\tau ^{22}-10\tau ^{16}\geq -30~.
\end{array}$$
Thus omitting the non-negative terms $80a^7\tau$ and
$10\tau ^{29}+10\tau ^{24}+10\tau ^{21}+5\tau ^{20}$ one can write

$$U_6\geq a^8(35-60/a^3-90/a^4-64/a^5-168/a^6-30/a^7)~.$$
The right-hand side is positive for $a\geq 2.4$, so $U_6>0$. Next,

$$\begin{array}{ccl}U_4&=&21a^{10}+60a^9\tau +(-60\tau ^3+30\tau ^2)a^8+
  (-12\tau ^6-60\tau ^4)a^7+
  (-84\tau ^{10}-36\tau ^7+30\tau ^6)a^6\\ \\
  &&+(-168\tau ^{15}-24\tau ^{11}+48\tau ^9)a^5+(84\tau ^{21}-96\tau ^{16}
  -12\tau ^{13}+24\tau ^{12})a^4\\ \\
  &&+(-462\tau ^{28}+90\tau ^{22}+42\tau ^{18}-30\tau ^{16})a^3+
  (-186\tau ^{29}-66\tau ^{24}+6\tau ^{21}+15\tau ^{20})a^2\\ \\
  &&+(60\tau ^{31}-36\tau ^{27}+12\tau ^{25})a+12\tau ^{34}+12\tau ^{31}
  +6\tau ^{30}~,\\ \\
  &\geq&21a^{10}+60a^9\tau -30\tau ^3a^8-72\tau ^4a^7-90\tau ^7a^6
  -144\tau ^{15}a^5\\ \\
  &&-84\tau ^{16}a^4-(330\tau ^{28}+30\tau ^{16})a^3-231\tau ^{24}a^2
  -24\tau ^{27}a\\ \\
  &\geq&a^{10}(21-\frac{30}{a^2}-\frac{90}{a^4}-\frac{84}{a^6}-\frac{231}{a^8})
  +\tau a^9(60-\frac{72}{a^2}-\frac{144}{a^4}
  -\frac{360}{a^6}-\frac{24}{a^8})~.\end{array}$$
Both expressions in brackets are positive for $a\geq 2.4$, so $U_4>0$.
In the same way we treat $U_2$:

  $$\begin{array}{ccl}
  U_2&=&7a^{12}+24a^{11}\tau +(-36\tau ^3+12\tau ^2)a^{10}
  +(-38\tau ^6-30\tau ^4)a^9+
  (14\tau ^{10}-34\tau ^7+15\tau ^6)a^8\\ \\
  &&+(-56\tau ^{15}+24\tau ^{11}+32\tau ^9)a^7+(196\tau ^{21}-16\tau ^{16}
  -28\tau ^{13}+16\tau ^{12})a^6\\ \\ 
  &&+(434\tau ^{28}+106\tau ^{22}-6\tau ^{18}-30\tau ^{16})a^5
  +(270\tau ^{29}-50\tau ^{24}-18\tau ^{21}+15\tau ^{20})a^4\\ \\
  &&+(-160\tau ^{31}-16\tau ^{27}+24\tau ^{25})a^3
  +(-88\tau ^{34}-4\tau ^{31}+12\tau ^{30})a^2\\ \\
  &&+(42\tau ^{38}-14\tau ^{36})a+14\tau ^{43}+7\tau ^{42}~,\\ \\
  &\geq&7a^{12}+24a^{11}\tau -24\tau ^3a^{10}-68\tau ^4a^9-19\tau ^7a^8\\ \\ &&
  -28\tau ^{13}a^6-36\tau ^{16}a^5
  -53\tau ^{21}a^4-152\tau ^{31}a^3-80\tau ^{34}a^2-14\tau ^{36}a\\ \\
  &\geq&a^{12}(7-\frac{24}{a^2}-\frac{19}{a^4})
  +\tau a^{11}(24-\frac{68}{a^2}-\frac{28}{a^5}-\frac{36}{a^6}
  -\frac{53}{a^7}-\frac{152}{a^8}
-\frac{80}{a^9}-\frac{14}{a^{10}})~,\end{array}$$
with positive values of the expressions in the brackets for $a\geq 2.4$.
Thus $U_2>0$. Finally, 
  
  $$\begin{array}{ccl}
    U_0&=&a^{14}+4a^{13}\tau +(-8\tau ^3+2\tau ^2)a^{12}+(-14\tau ^6-6\tau ^4)a^{11}+
  (22\tau ^{10}-10\tau ^7+3\tau ^6)a^{10}\\ \\
    &&+(32\tau ^{15}+16\tau ^{11}+8\tau ^9)a^9+(-44\tau ^{21}+24\tau ^{16}
    -12\tau ^{13}+4\tau ^{12})a^8\\ \\ &&+
  (-58\tau ^{28}-34\tau ^{22}-18\tau ^{18}-10\tau ^{16})a^7
    +(-46\tau ^{29}+26\tau ^{24}-14\tau ^{21}+5\tau ^{20})a^6\\ \\
    &&+(36\tau ^{31}+20\tau ^{27}+12\tau ^{25})a^5+
  (28\tau ^{34}-16\tau ^{31}+6\tau ^{30})a^4
    +(-22\tau ^{38}-14\tau ^{36})a^3\\ \\ &&+(-18\tau ^{43}+7\tau ^{42})a^2
    +16a\tau ^{49}+8\tau ^{56}~.
\end{array}$$
Consider the terms in $a^j$, $j=2$, $3$, $4$ and $5$.
For $a\geq 2.4$, it is true that 

$$a^5\tau ^{31}(36-16/a-(22\tau ^7+14\tau ^5)/a^2-18\tau ^{12}/a^3)\geq 0~,$$
so the sum of the mentioned terms is $\geq 0$.
Then we consider the terms in $a^j$, $j=6$, $7$, $8$ and $9$.
It is clear that for $a\geq 2.4$,

$$\begin{array}{l}
  8\tau ^9a^9+(-46\tau ^{29}+26\tau ^{24}-14\tau ^{21}+5\tau ^{20})a^6\geq
  \tau ^9a^9(8-(20\tau ^{20}+9\tau ^{12})/a^3)\geq \tau ^9a^9(8-29/a^3)\geq 0~,
  \\ \\
  16\tau ^{11}a^9+(-44\tau ^{21}+24\tau ^{16}-12\tau ^{13}+4\tau ^{12})a^8\geq
  \tau ^{11}a^9(16-(20\tau ^{10}+8\tau ^2)/a)\geq \tau ^{11}a^9(16-28/a)\geq 0\\ \\
       {\rm and}~~~\, 32\tau ^{15}a^9-(58\tau ^{28}+34\tau ^{22}+18\tau ^{18}
       +10\tau ^{16})a^7
  \geq \tau ^{15}a^9(32-120/a^2)\geq 0~,
\end{array}$$
so the sum of these terms is also $\geq 0$. The sum of the
terms in $a^{10}$, $\ldots$, $a^{14}$ is also $\geq 0$ for $a\geq 2.4$. To see
this it suffices to sum up the left- and right-hand sides of the inequalities

$$\begin{array}{l}a^{14}\geq 2.4^2a^{12}\geq 5.76\cdot \tau ^3a^{12}~,\\ \\  
0.1\cdot \tau a^{13}\geq 0.24\cdot \tau a^{12}\geq 0.24\cdot \tau ^3a^{12}~,\\ \\ 
3.5\cdot \tau a^{13}\geq 3.5\cdot 2.4^2\cdot \tau a^{11}\geq 20\cdot
\tau a^{11}\geq
(14\tau ^6+6\tau ^4)a^{11}~~~\, {\rm and}\\ \\ 
0.4\cdot \tau a^{13}\geq 0.4\cdot 2.4^3\cdot \tau a^{10}\geq 10\cdot
\tau ^7a^{10}\end{array}$$

All sums $U_j$ take non-negative values and $U_6>0$, so
$\partial (|U|^2)/\partial b\leq 0$ with equality only for $b=0$. 

Next, with $\tau =-q$ as above, we consider the sum of four consecutive terms
$(-\tau )^{j(j-1)/2}x^{-j}$, $j=4k+1$, $\ldots$, $4k+4$, $k\geq 2$,
of the series $G$:

$$\begin{array}{ccl}
  E^k&:=&\tau ^{2k(4k+1)}x^{-4k-1}-\tau ^{(2k+1)(4k+1)}x^{-4k-2}-
  \tau ^{(2k+1)(4k+3)}x^{-4k-3}+\tau ^{(2k+2)(4k+3)}x^{-4k-4}\\ \\
  &=&t_{\bullet}(x^3-\tau ^{4k+1}x^2-\tau ^{8k+3}x+\tau ^{12k+6})/x^{4k+4}~,~~~\,
  t_{\bullet}:=\tau ^{2k(4k+1)}~.\end{array}$$
One observes first that $x^3-\tau ^{4k+1}x^2-\tau ^{8k+3}x+\tau ^{12k+6}<0$
for $b=0$ and $x\leq -2.4$. 

Set $r:=\tau ^{4k+1}$, $s:=\tau ^{8k+3}$, $T:=\tau ^{12k+6}$, $A:=a^2+b^2$ and
$B:=a^2-b^2$. Using computer algebra one finds

$$\partial (|E^k|^2)/\partial b=(-2bt_{\bullet}/(a^2+b^2)^{4k+5})\cdot Z~,~~~\,
Z=A\cdot P-B\cdot Q+R~,~~~\, {\rm where}$$

$$\begin{array}{ccl}
  P&:=&(4k+1)A^2+(8kra+4kr^2+4ra+2r^2)A+(3+4k)s^2
  -(8k+6)sra~,\\ \\
  Q&:=&8ksA+4sb^2+8kTa+(8k+6)Tr~~~\, {\rm and}\\ \\
  R&:=&(16kT+18T)ab^2-8sa^4-14Ta^3-4Tra^2+8ksTa+4kT^2+8sTa+4T^2~.
\end{array}$$
Using $a\geq 2.4$, $T\leq s\leq r$, $|B|\leq A$ and $b^2\leq A$
one can write the following inequalities

$$2kA^3>8ksAB~,~~~\, 2A^3>8sa^4~~~\, {\rm and}~~~\, (2k-1)A^3>14Ta^3+4Tra^2$$
which imply

\begin{equation}\label{equineq1}
(4k+1)A^3>8ksAB+8sa^4+14Ta^3+4Tra^2~.
\end{equation}
Similarly from the inequalities

$$2kraA^2\geq 8kTaB~,~~~\, 2kraA^2\geq (8k+6)TrB~~~\,
{\rm and}~~~\, 4kraA^2\geq (8k+6)sraA$$
one deduces that

\begin{equation}\label{equineq2}
8kraA^2\geq 8kTaB+(8k+6)TrB+(8k+6)sraA~.
\end{equation}
Finally from (\ref{equineq1}), (\ref{equineq2}) and $4raA^2\geq 4sb^2B$
one concludes that $\partial (|E^k|^2)/\partial b\leq 0$ with equality
only for $b=0$.

Hence $G=U+\sum_{k=2}^{\infty}E^k$ and
$|G|\leq |U|+\sum_{k=2}^{\infty}|E^k|~(*)$, where the quantities $|U|$ and
$|E^k|$ are decreasing in~$b$ and there is equality in $(*)$ only for
$b=0$; we remind that for $b=0$, one has $U<0$ and $E^k<0$.
This proves the lemma.
\end{proof}

\subsection{Proof of part (3) of Theorem~\protect\ref{tmseparl}}

We set $x=a+bi$, $a\geq 3.2$, $b\geq 0$,
$\tau :=-q\in (0,1)$, $T_k:=-q^{2k-1}$ and
$Y(\tau ,T,x):=(1-T/x)(1-Tx)(1+T\tau /x)(1+T\tau x)$. Thus

$$\Theta^*:=\left( \prod_{j=1}^{\infty}(1-(-1)^j\tau^j)\right)
(1+1/x)\prod_{k=1}^{\infty}Y(\tau ,T_k,x)~.$$
It suffices to prove part (3) of Theorem~\ref{tmseparl}
only for $\tau \in [0.75,1)$, because the second spectral value is
  $\bar{q}=-0.78\ldots$ and for $q\in (\bar{q}_2,0)$,
  there are no complex zeros with positive real part. 
We modify the method applied to the proofs of parts (1) and (2) as follows:

\begin{lm}\label{lmpr(3)}
  For $\tau \in [0.75,1)$ and $a$, $b$ as above,

  (1) $|\Theta^*(-\tau ,x)/x^2|$ is an increasing
  function in $b\geq 0$;

  (2) $|G(-\tau ,x)/x^2|$ is majorized by a decreasing function in $b$ for
  $b\geq 0$ which coincides with $|G(-\tau ,x)/x^2|$ for $b=0$.  
  \end{lm}

Thus if for $x=a\geq 3.2$, one has $\theta =\Theta^*-G=0$, then
for $x=a+bi$, $b>0$, one has $\theta \neq 0$ and for $\varepsilon >0$
small enough, Re$x=a-\varepsilon$ is a right separating line. 

\begin{proof}[Proof of part (1) of Lemma~\ref{lmpr(3)}]
We need the following lemma:

\begin{lm}\label{lmY*}
  For
  $0<T\leq \tau <1$, one has
  $\partial |Y(\tau ,T,x)|^2/\partial b\geq 0$ with equality only for $b=0$.
\end{lm}

Lemmas~\ref{lmY*} and \ref{lmYcirc} are proved after part (1) of  Lemma~\ref{lmpr(3)}. 
For $T_1$ and $T_2$, we prove a stronger statement:

\begin{lm}\label{lmYcirc}
  For $x=a+bi$, $a\geq 3.2$, $b\geq 0$, $\tau\in [0.75,1)$ and $T=\tau$ or 
    $T=\tau^3$, one has
  $\partial (|Y(\tau ,T,x)|^2/(a^2+b^2))/\partial b\geq 0$ with equality
  only for $b=0$.
\end{lm}
We need also a third lemma:

\begin{lm}\label{lmsqrtx}
  Set $x=a+bi$, $a\geq 3.2$. Then the quantity $r:=|x^{1/2}+1/x^{1/2}|$
  is an increasing function in $b\geq 0$.
\end{lm}

\begin{proof}
  We show that $\partial r^4/\partial b\geq 0$. A direct computation yields
$r^2=|x+1/x+2|$, 
  $$\begin{array}{ccl}r^4=|x+1/x+2|^2&=&
    (a+bi+(a-bi)/(a^2+b^2)+2)(a-bi+(a+bi)/(a^2+b^2)+2)~~~\, {\rm and}\\ \\
\partial r^4/\partial b&=&2b(a^2+b^2+2a+1)(a^2+b^2-2a-1)/(a^2+b^2)^2~.
 \end{array}
    $$
  For $a\geq 3.2$ and $b\geq 0$, the latter product is non-negative; it equals $0$ only for $b=0$.
  \end{proof}
  
Lemmas~\ref{lmY*}, \ref{lmYcirc} and \ref{lmsqrtx} imply that 
all quantities $|\sqrt{x}+1/\sqrt{x}|=|\sqrt{x}(1+1/x)|$,
$|Y(\tau ,T_1,x)/\sqrt{x}|=|\sqrt{x}(Y(\tau ,T_1,x)/x)|$, $|Y(\tau ,T_2,x)/x|$ and $|Y(\tau ,T_k,x)|$, $k\geq 3$,
are increasing in $b\geq 0$. This implies part~(1) of Lemma~\ref{lmpr(3)}. 
\end{proof}

\begin{rem}\label{rempoly}
 {\rm In the proofs of Lemmas~\ref{lmY*} and \ref{lmYcirc}
we use the following notation:

$$\begin{array}{ccl}
  A&:=&T^2\tau (1-\tau )(1-T^2\tau )\geq 0~,~~~\,
  B~:=~T(\tau^2+1)(T^4\tau^2+1)\geq 0~,\\ \\
  C&:=&8T^5\tau^3-4T^5\tau^2+8T^3\tau^3+8T^3\tau -4T\tau^2+8T\tau ~,\\ \\ 
  D&:=&2(1-\tau )(1-T^2\tau )(T^4\tau^2+T^2\tau^2+T^2\tau +T^2+1)\geq 0~,\\ \\ 
  U&:=&T^4\tau^4-8T^4\tau^3+T^4\tau^2-8T^2\tau^3+16T^2\tau^2
  -8T^2\tau+\tau^2-8\tau+1~,\\ \\ 
  V&:=&4T(1-\tau )(1-T^2\tau)(T^4\tau^2+T^2\tau^2-3T^2\tau+T^2+1)\\ \\ 
  &=&4T(1-\tau )(1-T^2\tau )(T^2(1-\tau )^2+(1-T^2\tau )^2/2+(T^4\tau^2+1)/2)
  \geq 0~,\\ \\ 
  W&:=&T^4\tau^2+T^2\tau^2-4T^2\tau +T^2+1~=~(1-T^2\tau)^2+T^2(1-\tau)^2\geq 0~,
  \end{array}$$ 

$$\begin{array}{ccl}L&:=&T^2\tau (T^4\tau^2+T^2\tau^2-2T^2\tau +T^2+1)~,\\ \\ 
 M&:=&(-T^8\tau^4-4T^6\tau^4+8T^6\tau^3)+(-4T^6\tau^2-T^4\tau^4+8T^4\tau^3)\\ \\ 
 &&+(-2T^4\tau^2+8T^4\tau-T^4)-4T^2(1-\tau )^2-1~,\\ \\
 R&:=&2T(1-\tau )(1-T^2\tau )(T^4\tau^2+T^2\tau^2-7T^2\tau +T^2+1)\\ \\
 &=&2T(1-\tau )(1-T^2\tau )((1+T^2\tau )^2+T^2(1+\tau )^2-3T^2\tau )\geq
 -6T^3\tau (1-\tau )(1-T^2\tau )~,\\ \\
 S&:=&-8T^2(1-\tau )^2(1-T^2\tau )^2-(1-T^4\tau ^2)^2-
 T^4(1-\tau^2)^2-22T^4\tau^2\leq 0~,\\ \\ 
 Q&:=&4T(1-\tau )(1-T^2\tau )(T^4\tau^2+T^2\tau^2-5T^2\tau+T^2+1)\\ \\
 &=&4T(1-\tau )(1-T^2\tau )((1+T^2\tau )^2+T^2(1+\tau )^2-T^2\tau )\geq -4T^3
 \tau (1-\tau )(1-T^2\tau )~,\\ \\
 H&:=&2T^2((T^4\tau^2+1)(-\tau^2+6\tau -1)+6T^2\tau (1-\tau )^2)
 \geq 2T^2(-\tau^2+6\tau -1)~.
  \end{array}$$
The polynomials $A$, $B$, $C$, $D$, $V$, $W$ and $L$
are positive for $0<T\leq \tau <1$. 
For $C$ this follows from $C>-4T^5\tau^2-4T\tau^2+8T\tau \geq 0$, 
 for $L$ it results from $T^2\tau^2-2T^2\tau+T^2=T^2(1-\tau)^2$. }
 \end{rem}

\begin{proof}[Proof of Lemma~\ref{lmY*}]
  Using computer algebra (MAPLE) one finds that

  $$\partial |Y(\tau,T,x)|^2/\partial b=2bTY^*/(a^2+b^2)^3~,~~~\,
  Y^*=\sum_{j=0}^4Y_jb^{2j}~,~~~\, {\rm where}~~~\, Y_8=2T^3\tau^2>0~,$$
  $$Y_6=8T^3\tau^2a^2+2aA+B~,~~~\, 
    Y_4=12T^3\tau^2a^4+6a^3A+3a^2B~.$$
    As $A\geq 0$ and $B\geq 0$, see Remark~\ref{rempoly},
    one obtains $Y_6\geq 0$ and $Y_4\geq 0$.
  Next,

  $$Y_2=8T^3\tau^2a^6+A(6a^5-8a^3)+3a^4B+(C-4T^5\tau^4-4T)a^2+
  aD-B~.$$
  Clearly $C>0$, $D\geq 0$ (see Remark~\ref{rempoly}) 
  and for $a\geq 3.2$, 
  $(6a^5-8a^3)A\geq 0$. Also

  $$8T^3\tau^2a^6\geq 0~~~\, {\rm and}~~~\,
  (3a^4-1)B-(4T^5\tau^4+4T)a^2\geq (3a^4-1-8a^2)T>0~,$$
  so $Y_2>0$. 
  In a similar way we treat $Y_0$:

  $$\begin{array}{ccl}Y_0&=&
    2T^3\tau^2a^8+(2a^7-8a^5)A+a^6B+
    (C-32T^3\tau^2-4T^5\tau^4-4T)a^4\\ \\
&&+(2(1-\tau)(1-T^2\tau)W-6A)a^3-TUa^2-(4T^2(1-\tau)+4A)a-2T^3\tau^2~. 
  \end{array}$$
  The sum of the terms containing a factor $A$ is
  $Aa^7(2-8/a^2-6/a^4-4/a^6)\geq 0$.

  As
  $W\geq 0$ (see Remark~\ref{rempoly}), one has
  $2(1-\tau )(1-T^2\tau )Wa^3\geq 0$. 
  The coefficient $-TU$ of $a^2$ is
  $\geq -T(T^4\tau^4+T^4\tau^2+16T^2\tau^2+\tau^2+1)\geq -20T$.

  One has also
  $-4T^2(1-\tau)a\geq -4T\tau(1-\tau)a\geq -Ta$. Therefore

  $$a^6B-4Ta^4-TUa^2-4T^2(1-\tau)a\geq Ta^6(1-4/a^2-20/a^4-1/a^5)\geq 0~.$$
  The sum of the terms of $Y_0$ not involved in any of the above inequalities is

  $$2T^3\tau^2a^8+Ca^4-(32T^3\tau^2+4T^5\tau^4)a^4-2T^3\tau^2
  \geq T^3\tau^2a^8(2-36/a^4-2/a^8)\geq 0$$
  from which $Y_0>0$ follows. This proves the lemma.
  
  \end{proof}

\begin{proof}
[Proof of Lemma~\ref{lmYcirc}]  
  By means of computer algebra (MAPLE) and using the notation from the
  proof of Lemma~\ref{lmY*} one finds that

  $$\begin{array}{l}
    \partial (|Y(\tau,T,x)|^2/(a^2+b^2))/\partial b=2bY^{\circ}/(a^2+b^2)^4~,~~~\,
  Y^{\circ}=\sum_{j=0}^4Y_j^{\circ}b^{2j}~,\\ \\ {\rm where}~~~\,
  Y_8^{\circ}=T^4\tau^2>0~~~\, {\rm and}~~~\, Y_6^{\circ}=4T^4\tau^2a^2>0~.\end{array}$$
  The next term is

  $$Y_4^{\circ}=6T^4\tau^2a^4+4a^2L+TaD
    -T^8\tau^4-T^4\tau^4-2T^4\tau^2-T^4-1~.$$
    The summands containing a factor which is a power of $a$ are all positive, so it suffices to
    show that $Y_4^{\circ}>0$ for $a=3.2$. It can be checked numerically that for
    $T=\tau$ or $T=\tau^3$,  and $\tau \in [0.75,1)$, the three
      corresponding polynomials in $\tau$ are positive-valued.
     Next,
      $$\begin{array}{ccl}
    Y_2^{\circ}&=&4T^4\tau^2a^6+8a^4L+Va^3+
    2a^2M+
    2TaD-2TB~.
  \end{array}$$
     One checks numerically that for $T=\tau$ or $T=\tau^3$ 
     and for $a=3.2$,
one has $4T^4\tau^2a^4+8a^2L+Va+2M>0.35>0$. The terms in this sum, which are
multiplied by powers of $a$, are positive. Hence the sum is minimal for
$a=3.2$. Given that $D>0$, the quantity $Y_2^{\circ}$ is minimal for $a=3.2$.
Another numerical check shows that $Y_2^{\circ}>0$ for $a=3.2$
(hence for $a\geq 3.2$) and for 
$T=\tau$ or $T=\tau^3$. 

Now we consider

$$Y_0^{\circ}=T^4\tau^2a^8+4La^6+Ra^5+Sa^4+Qa^3+Ha^2-6TAa-3T^4\tau^2~.$$
We show first that $a^2(H-6TA/a-3T^4\tau^2/a^2)>0$.
Indeed, for $\tau\in [0.75,1)$
  and $a\geq 3.2$, one has $H\geq 2T^2(-\tau^2+6\tau-1)\geq 5T^2$, see
  Remark~\ref{rempoly}, 
  while $6TA/a<2T^3\tau<2T^2$ and $3T^4\tau^2/a^2<T^2/3$. Set
  $\tilde{U}:=T^4\tau^2a^8+4La^6+Ra^5+Sa^4+Qa^3$. From Remark~\ref{rempoly}
  becomes clear that

  $$\tilde{U}\geq
  T^4\tau^2a^8+4La^6-6T^3\tau(1-\tau)(1-T^2\tau)a^5+Sa^4-
  4T^3\tau(1-\tau)(1-T^2\tau)a^3~.$$
  The first two terms in the right-hand side
  are positive whereas the last three are not. Therefore $\tilde{U}$ 
  is minimal for $a=3.2$. One checks numerically that the two polynomials
$\tilde{U}|_{a=3.2,T=\tau}$ and $\tilde{U}|_{a=3.2,T=\tau^3}$
  are positive-valued for $\tau\in [0.75,1)$
    which implies $Y_0^{\circ}>0$ and hence
    $Y^{\circ}>0$. This proves the lemma.
\end{proof}

\begin{proof}[Proof of part (2) of Lemma~\ref{lmpr(3)}]
 We represent $G(-\tau,x)/x^2$ as a sum

$$\begin{array}{ccl}G(-\tau,x)/x^2&=&
  1/x^3-\tau/x^4-\tau^3/x^5+\tau^6/x^6+\tau^{10}/x^7-\cdots  
  =\sum_{j=1}^{\infty}G_j^*~,~~~\, {\rm where}\\ \\ 
  G_j^*(\tau,x)&:=&
  \tau^{(2j-2)(4j-3)}/x^{4j-1}-\tau^{(2j-1)(4j-3)}/x^{4j}-\tau^{(2j-1)(4j-1)}/x^{4j+1}+
  \tau^{2j(4j-1)}/x^{4j+2}~.
  \end{array}$$
\begin{lm}\label{lmGj}
  (1) For $x\in \mathbb{R}$, $x\geq 3.2$, $\tau\in (0,1)$, one has $G_j^*>0$.
  
  (2) For {\rm Re}$x\geq 3.2$ and $\tau\in (0,1)$ fixed, each quantity
  $|G_j^*|$ is majorized by a
  decreasing function in $|{\rm Im}x|$.

  (3) For {\rm Re}$x\geq 3.2$ and $\tau\in (0,1)$ fixed,
  the quantity $|G(-\tau,x)/x^2|$ is majorized by a
  decreasing function in $|{\rm Im}x|$.
\end{lm}

\begin{proof}[Proof of Lemma~\ref{lmGj}]
  Part (1). It is clear that
  $G_j^*\geq (\tau^{(4j-3)(2j-1)}/x^{4j-1})(1-1/x-1/x^2)$. The second factor
  is positive for $x\geq 3.2$ which proves part (1) of the lemma.

  Part (2). We consider first
  $G_1^*=1/x^3-\tau/x^4-\tau^3/x^5+\tau^6/x^6+\tau^{10}/x^7$.
  We set
  $x:=a+bi$, $a\geq 3.2$, $b\geq 0$, by means of computer algebra (MAPLE) we 
  find first
  $|G_1^*|^2=G_1^*|_{x=a+bi}\cdot G_1^*|_{x=a-bi}$ and then 

  $$\begin{array}{ccl}\partial |G_1^*|^2/\partial b&=&
    -2bG^1/(a^2+b^2)^7~,~~~\, {\rm where}~~~\, 
    G^1:=G^1_0+G^1_2b^2+G^1_4b^4~,\\ \\
    G^1_0&:=&3a^6-8\tau a^5+(4\tau^2-12\tau^3)a^4+(10\tau^4+18\tau^6)a^3+
    (5\tau^6-14\tau^7)a^2-12\tau^9a+6\tau^{12}~,\\ \\
    G^1_2&:=&9a^4-16\tau a^3+(8\tau^2-4\tau^3)a^2+
    (10\tau^4-30\tau^6)a+10\tau^7+5\tau^6~,\\ \\
    G^1_4&:=&9a^2-8\tau a+4\tau^2+8\tau^3~.\end{array}$$
  Clearly $G^1_4\geq 0$ for $a\geq 3.2$, $\tau \in (0,1)$. As for $a\geq 3.2$,
  $\tau \in (0,1)$, one has 

  $$G^1_2\geq a^4G^1_{2,0}~,~~~\, G^1_{2,0}:=9-16/a+4t^2/a^2-20/a^3$$
  with $G^1_{2,0}>0$, one concludes that
  $G^1_2>0$. To prove that $G^1_0>0$ one observes first that

  $$4\tau^6a^3+(5\tau^6-14\tau^7)a^2-12\tau^9a+6\tau^{12}\geq
  \tau^6a^3(4-9/a-12/a^2)~,$$
  where $4-9/a-12/a^2>0$ for $a\geq 3.2$. Hence

  $$\begin{array}{ccl}G^1_0&\geq&
    3a^6-8\tau a^5+(4\tau^2-12\tau^3)a^4+
    (10\tau^4+14\tau^6)a^3\geq a^6G^{1,0}~,\\ \\
    G^{1,0}&:=&3-8/a-8\tau^3/a^2+(10\tau^4+14\tau^6)/a^3~.\end{array}$$
  One finds numerically that the minimal value of $G^{1,0}|_{a=3.2}$ for
  $\tau\in [0,1]$ is $>0.3$, and one computes

  $$\partial G^{1,0}/\partial a=2((4a^2-21\tau^6)+\tau^3(8a-15\tau))/a^4~.$$
  Obviously $\partial G^{1,0}/\partial a>0$ for $a\geq 3.2$, $\tau\in (0,1)$,
  so $G^1_0>0$ and hence $G^1>0$. This proves the claim of the lemma about
  the quantity $|G_1^*|$.

  To prove part (2) for any $j\in \mathbb{N}^*$, $j\geq 2$,
  we set $\alpha :=4j-4$, $\Lambda :=(2j-2)(4j-3)-3\alpha =8j^2-26j+18$
  and we represent $G_j^*$ in the form

  $$
  G_j^*(\tau,x)=
    \frac{\tau^{\Lambda}}{x^{4j-4}}\left( \frac{\tau^{3\alpha}}{x^3}-
    \frac{\tau^{4\alpha +1}}{x^4}-\frac{\tau^{5\alpha +3}}{x^5}+
    \frac{\tau^{6\alpha +6}}{x^6}\right) =\frac{\tau^{\Lambda}}{x^{4j-4}}
    G_1^*(\tau,x/\tau^{\alpha})~.
    $$
    If Re$x\geq 3.2$, then Re$(x/\tau^{\alpha})\geq 3.2$, so $|G_1^*(\tau,x/\tau^{\alpha})|$ is
    majorized by a
    decreasing function in $|{\rm Im}x|$. For fixed $\tau\in (0,1)$,
    the function $|\tau^{\Lambda}/x^{4j-4}|$ decreases
    as $|{\rm Im}x|$ increases. Therefore $|G_j^*(\tau,x)|$ is
     majorized by a
    decreasing function in $|{\rm Im}x|$.

  Part (3). Clearly $|G(-\tau,x)/x^2|\leq \sum_{j=1}^{\infty}|G_j^*|$ with equality
  for $x\in \mathbb{R}$, $x\geq 3.2$, see part (1). That's why part (3)
  results from part~(2). 

\end{proof}

\end{proof}

\section{Proof of Theorem~\protect\ref{tmimprove}\protect\label{secprtmimprove}}

We denote by $x_j$ the real zeros of $\theta (q,.)$ for $q\in (-1,0)$. For
$q\in (\bar{q}_1,0)$, all its zeros are real and for $q\in [-0.108,0)$, 
they satisfy the following
inequalities, see \cite[Figure~3]{KoPRSE2}:

\begin{equation}\label{equstring}
  \begin{array}{clc}
  \cdots &<x_{4k+8}<x_{4k+6}<qx_{4k+7}<qx_{4k+5}<x_{4k+4}<x_{4k+2}<qx_{4k+3}<0& \\ \\
  &0<x_{4k-1}<x_{4k+1}<qx_{4k+2}<qx_{4k+4}<x_{4k+3}<x_{4k+5}<qx_{4k+6}<&\cdots ~,
  \end{array}
  \end{equation}
$k=0$, $1$, $2$, $\ldots$. Hence odd zeros are positive and even zeros
are negative.

Part (1) of the theorem is proved for even zeros in \cite{KoPRSE2}, see
Theorem~1.4 therein, and for odd zeros only if $k$ is sufficiently large.
We deduce here part (1) of Theorem~\ref{tmimprove}, both
for even and odd zeros, from its part~(2).

It is shown in \cite{KoPRSE2} that for $q\in (-1,0)$,
when zeros of
$\theta$ coalesce, then it is $x_{4k+4}$ with $x_{4k+2}$ and
$x_{4k+3}$ with $x_{4k+5}$ that do so; in particular,
the zero $x_1$ is simple for any $q\in (-1,0)$.
Besides, as $j\rightarrow \infty$, one has
$x_j+1/q^j\rightarrow 0$. That's why we say that the zeros
$x_{4k+4}$ and $x_{4k+2}$
correspond to the interval
$(-1/q^{4k+4},-1/q^{4k+2})$ and the zeros $x_{4k+3}$ and $x_{4k+5}$
correspond to the interval
$(-1/q^{4k+3},-1/q^{4k+5})$. 

\begin{prop}\label{propdoublezeros}
  (1) For $s\in \mathbb{N}^*$, one has $\theta (q,-q^{-2s})>0$ and
  $\theta (q,-q^{-2s-1})<0$ (the latter inequality holds true also for $s=0$,
  see {\rm \cite[Proposition~4.5]{KoPRSE2}}).

  (2) Hence the double zero corresponding to the spectral value
  $\bar{q}_{2\ell -1}$
(resp. $\bar{q}_{2\ell +2}$) belongs to the interval 
$(-1/q^{4\ell},-1/q^{4\ell -2})$ (resp. $(-1/q^{4\ell +3},-1/q^{4\ell +5})$).
\end{prop}

The proofs of Proposition~\ref{propdoublezeros}, Lemma~\ref{lmV} and
Proposition~\ref{propintervals} are given at the end of the section.


\begin{lm}\label{lmV}
  Set $V:=1+qx+q^3x^2$. For $q\in (-1,-0.84]$ and $|x|>2.2$, 
        one has $V<0$.
\end{lm}

With the help of the lemma we prove the following proposition (we remind
that we consider $q$ as decreasing from $0$ to~$-1$):

\begin{prop}\label{propintervals}
  (1) Suppose that for $q\in (-1,0)$ and for $x_*\in (-1/q^{4k+3},-1/q^{4k+5})$,
  $k\in \mathbb{N}^*$, one has $\theta (q,x_*)=0$. Then
  $\theta (q,x_*/q^3)<0$. Hence if $x_*$ is a double zero of
  $\theta (\bar{q}_{2k+2},.)$, then the two zeros $x_{4k+8}$ and $x_{4k+6}$
  which correspond
  to the interval
  $(-1/q^{4k+8},-1/q^{4k+6})$ have not coalesced yet.
  So $-\bar{q}_{2k+2}<-\bar{q}_{2k+3}$. The latter inequality holds true also
  for $k=0$. 
  \vspace{1mm}

  (2) Suppose that for $q\in (-1,0)$ and for
  $x_{\diamond}\in (-1/q^{4k},-1/q^{4k-2})$,
  $k\in \mathbb{N}^*$, one has $\theta (q,x_{\diamond})=0$. Then
  $\theta (q,x_{\diamond}/q)>0$. Hence if $x_{\diamond}$ is a double zero of
  $\theta (\bar{q}_{2k-1},.)$, then the two zeros
  $x_{4k-1}$ and $x_{4k+1}$ which correspond to the interval
  $(-1/q^{4k-1},-1/q^{4k+1})$ have not coalesced yet. Thus
  $-\bar{q}_{2k-1}<-\bar{q}_{2k}$.
  \end{prop}

The proposition implies the string of inequalities
$0<-\bar{q}_j<-\bar{q}_{j+1}<1$, $j=1$, $2$, $\ldots$
(this is part (2) of Theorem~\ref{tmimprove}). Part (1) of the theorem
is proved for spectral values $\bar{q}_{2k-1}$ in \cite{KoPRSE2}. For
$\bar{q}_{2k+2}$, it results from the zeros corresponding to the interval
$(-1/q^{4k+3},-1/q^{4k+5})$ coalescing before the ones corresponding to
$(-1/q^{4k+8},-1/q^{4k+6})$ and the latter coalescing before the ones
corresponding to $(-1/q^{4k+7},-1/q^{4k+9})$. This is true also for $k=0$. 

Part (3) of the theorem follows from the fact that
if $q\in (\bar{q}_{j+1},\bar{q}_j)$, then
exactly $j$ times $q$, when decreasing from $0$ to $-1$,
has passed through a spectral
value. At each such passage two real zeros coalesce and form a complex
conjugate pair.

\begin{proof}[Proof of Proposition~\ref{propdoublezeros}]
Set $v:=-q$, so $v\in (0,1)$. For
$\bar{q}_{2\ell -1}$ and with $s\in \mathbb{N}^*$,
the proposition follows from the equality

$$\theta (q,-q^{-2s})=\theta (-v,-v^{-2s})=
\sum_{j=0}^{\infty}(-1)^{j(j+1)/2+j}v^{j(j+1)/2-2sj}=
\sum_{j=0}^{\infty}(-1)^{j(j+3)/2}v^{j(j+1)/2-2sj}~.$$
Indeed, the first $4s$ terms cancel (the first with the $4s$th, the second with
the $(4s-1)$st etc.). The sequence $((-1)^{j(j+3)/2})_{j\geq 4s}$ equals
$1$, $1$, $-1$, $-1$, $1$, $1$, $-1$, $\ldots$, so

$$\theta (-v,-v^{-2s})=v^{2s}+v^{4s+1}-v^{6s+3}-v^{8s+6}+v^{10s+10}+v^{12s+15}-\cdots$$
which is the sum of the two Leibniz series $v^{2s}-v^{6s+3}+v^{10s+10}-\cdots$ and
$v^{4s+1}-v^{8s+6}+v^{12s+15}-\cdots$ with positive initial terms hence
with positive sums. Thus $\theta (q,-q^{-2s})>0~(**)$. For small values of $|q|$,
the zeros of $\theta$ are close to the quantities $-1/q^j$,
see~\cite{KoDBAN15}.
Besides, one has the inequalities (\ref{equstring}) which together with
the inequality $(**)$ (with $s=2\ell -1$ and $s=2\ell$)
imply that the zeros $x_{4\ell -2}$ and $x_{4\ell}$
belong to the interval $I:=(-1/q^{4\ell},-1/q^{4\ell-2})$. As $(**)$ holds true for
all $q\in (-1,0)$, these zeros remain in the interval $I$ until they coalesce.

Consider now the spectral value $\bar{q}_{2\ell +2}$. 
One finds that

$$\theta (q,-q^{-2s-1})=\theta (-v,v^{-2s-1})=
\sum_{j=0}^{\infty}(-1)^{j(j+1)/2}v^{j(j+1)/2-(2s+1)j}~.$$
The first $4s+2$ terms cancel and the sequence $((-1)^{j(j+1)/2})_{j\geq 4s+2}$
equals $-1$, $1$, $1$, $-1$, $-1$, $1$, $\ldots$. Thus

$$\theta (-v,v^{-2s-1})=-v^{2s+1}+v^{4s+3}+v^{6s+6}-v^{8s+10}
-v^{10s+15}+v^{12s+21}-\cdots$$
We set $\varphi_k(v):=\sum_{j=0}^{\infty}(-1)^kv^{kj+j(j-1)/2}=
1-v^k+v^{2k+1}-v^{3k+3}+v^{4k+6}-\cdots$. One has

$$\begin{array}{ccl}
  \theta (-v,v^{-2s-1})&=&-v^{2s+1}+v^{6s+6}-v^{10s+15}+\cdots
  +v^{4s+3}-v^{8s+10}+v^{12s+21}-\cdots \\ \\
  &=&-v^{2s+1}(1-v^{4s+5}+v^{8s+14}-\cdots)+v^{4s+3}(1-v^{4s+7}+v^{8s+18}-\cdots )~.
\end{array}$$
We set $w:=v^4\in (0,1)$. Hence
$\theta (-v,v^{-2s-1})=-w^{(2s+1)/4}\varphi_{s+5/4}(w)+w^{(4s+3)/4}\varphi_{s+7/4}(w)$.
We show that $\theta (-v,v^{-2s-1})<0$, i.~e.

\begin{equation}\label{equw}
  -w^{(-s-1)/2}\varphi_{s+5/4}(w)+\varphi_{s+7/4}(w)<0~.
\end{equation}
We use the functions $\xi_k(w):=1/(1+w^k)$.
It follows from \cite[Proposition~16]{KoBSM1}
that for $w\in (0,1)$ and $k>1$, one has $\xi_{k-1}(w)<\varphi_k(w)<\xi_k(w)$.
In this way 

$$
  -w^{(-s-1)/2}\varphi_{s+5/4}+\varphi_{s+7/4}<-w^{(-s-1)/2}\xi_{s+1/4}+\xi_{s+7/4}
  =\Xi (w)/((1+w^{s+1/4})(1+w^{s+7/4}))~,$$
where $\Xi :=-w^{(-s-1)/2}-w^{s/2+3/4}+1+w^{s+1/4}$. Given that

$$\begin{array}{ccl}\Xi &=&
  -(w^{(-s-1)/2}-w^{(-s+1)/2})-w^{(-s+1)/2}-w^{s/2+3/4}+1+w^{s+1/4}\\ \\ 
&<&-w^{(-s+1)/2}-w^{s/2+3/4}+1+w^{s+1/4}=(1+w^{s+1/4})(1-w^{(-s+1)/2})<0~,\end{array}$$
this implies (\ref{equw}). As in the case of $\bar{q}_{2\ell -1}$,
we conclude that
for small values of $|q|$, the zeros $x_{4\ell +3}$ and $x_{4\ell +5}$ 
belong to the interval $(-1/q^{4\ell +3},-1/q^{4\ell +5})$
and remain in it until they coalesce.
  
\end{proof}

\begin{proof}[Proof of Lemma~\ref{lmV}]
  It suffices to prove the lemma for $x<0$.
  Set $v=|q|$ and $y:=|x|$, so $v\in [0.84,1)$ and $y>2.2$.
  We set $h(v,y):=1+vy-v^3y^2$.
  As
  $$\partial h/\partial v=y(1-3v^2y)\leq y(1-3\cdot 0.84^2\cdot 2.2)
  =-3.65696y<0~,$$
  for $y$ fixed, the function $h$ is maximal for $v=0.84$. Set
  $h_0(y):=h(0.84,y)$. Since

  $$h_0'=0.84\cdot (1-2\cdot 0.84^2y)\leq
  0.84\cdot (1-2\cdot 0.84^2\cdot 2.2)=-1.7678976\ldots <0~,$$
  the maximal value of $h$ is $h(0.84,2.2)=-0.02$, so
  $V\leq -0.02\ldots <0$.
    
  \end{proof}

\begin{proof}[Proof of Proposition~\ref{propintervals}]
It is true that
  $-\bar{q}_1=0.72\ldots <-\bar{q}_2=0.78\ldots <-\bar{q}_3=0.84\ldots$ 
(the values of $-\bar{q}_j$ for $j\leq 8$ can be found in~\cite{KoPRSE2}).
So we need to prove the inequality $-\bar{q}_{2k+2}<-\bar{q}_{2k+3}$
only for $k\geq 1$.
All statements
  of part (1) follow from

  $$\theta (q,x_*/q^3)=1+x_*/q^2+x_*^2/q^3+(x_*^3/q^3)\theta(q,x_*)=V(q,x_*/q^3)+
  (x_*^3/q^3)\theta(q,x_*)~,$$
  with $V(q,x_*/q^3)<0$, see Lemma~\ref{lmV}. Indeed, the partial theta function
satisfies the differential equation 

\begin{equation}\label{equpartder}2q\cdot \partial \theta /\partial q
=2x\cdot \partial \theta /\partial x+x^2\cdot \partial ^2\theta /\partial x^2~,
\end{equation}
so when
$q$ decreases from $0$ to $-1$, local minima go up and local maxima go down.
Thus if $x_*$ is a double zero of $\theta (\bar{q}_{2k+2},.)$ and
$\theta (q,x_*/q^3)<0$, then at the local minimum of $\theta$ on the interval
$(-1/q^{4k+8},-1/q^{4k+6})$ the value of $\theta$ is negative, so a double zero
of $\theta$ on this interval occurs for $q=\bar{q}_{2k+3}<\bar{q}_{2k+2}$. Hence
$-\bar{q}_{2k+2}<-\bar{q}_{2k+3}$.

The same reasoning allows to deduce part (2) from the equalities
and inequality
  $\theta (q,x_{\diamond}/q)=1+qx_{\diamond}\theta (q,x_{\diamond})=1>0$,
  see~(\ref{equqx}). 
  \end{proof}

\section{Proofs of Propositions~\protect\ref{propsecond}, 
  \protect\ref{propincr}, \protect\ref{prop2kincrease} and
  Lemma~\protect\ref{lmkgeq1}
  \protect\label{secprpropsecond}}

We prove first two lemmas which are of independent interest.
In order to avoid ambiguity we sometimes use the notation
$\theta (u,v)=\sum_{j=0}^{\infty}u^{j(j+1)/2}v^j$, so 
$(\partial \theta /\partial v)(q,x)=\sum_{j=0}^{\infty}jq^{j(j+1)/2}x^{j-1}$ etc.

\begin{lm}\label{lm1/2}
  For $q\in (0,1)$ and $x=-q^{-1/2}$, one has $\partial \theta /\partial q<0$.
\end{lm}

\begin{proof}
  One considers the function

  $$\nu (q):=\theta (q,-q^{-1/2})=
  1-q^{-1/2}+q^2-q^{-9/2}+q^8-\cdots =\sum _{k=0}^{\infty}(-1)^kq^{k^2/2}~.$$
  This
  function is decreasing. Indeed, it is shown in
  \cite[Chapter~1, Problem~56]{PoSz} that

 $$   1+2\sum _{k=1}^{\infty}(-1)^kq^{k^2}=\prod_{k=1}^{\infty}((1-q^k)/(1+q^k))~,$$
  where all factors are decreasing functions. On the other hand

  $$0>\nu '(q)=\partial \theta (q,x)/\partial q|_{x=-q^{-1/2}}+
  ((\partial \theta (q,x)/\partial x)|_{x=-q^{-1/2}})\cdot (1/2q^{3/2})~.$$
    The latter term is positive for $x\geq -q^{-1/2}$, because
    $\partial \theta /\partial x>0$ for $x>-q^{-1}$ hence for
    $x=-q^{-1/2}$, see
    \cite[Part~(3) of Proposition~4.6]{KoPRSE2}. Thus 
$\partial \theta (q,x)/\partial q|_{x=-q^{-1/2}}<0$. 
\end{proof}

\begin{lm}
  The following two equalities hold true:

  \begin{equation}\label{equ2equ}
    x(\partial ^2\theta /\partial v^2)(q,x)=
  2q^2(\partial \theta /\partial u)(q,qx)~~~\, {\rm and}~~~\, 
  x^2(\partial ^4\theta /\partial v^4)(q,x)=
  4q^5(\partial ^2\theta /\partial u^2)(q,q^2x)~.
  \end{equation}
\end{lm}

\begin{proof}
  We set $A_j(u,v):=u^{j(j+1)/2}v^j$. One checks directly that
  for $j\in \mathbb{N}$, 

  $$\begin{array}{lc}
    x(\partial^2A_j/\partial v^2)(q,x)=j(j-1)q^jA_{j-1}(q,x)&{\rm and}\\ \\
    (\partial A_{j-1}/\partial u)(q,qx)=(j(j-1)/2)q^{j-2}A_{j-1}(q,x)~,&{\rm so}
    \\ \\ 2q^2(\partial A_{j-1}/\partial u)(q,qx)=j(j-1)q^jA_{j-1}(q,x)
  \end{array}$$
  which proves the first equality of the lemma. Next,

  $$\begin{array}{lc}
    x^2(\partial ^4A_j/\partial v^4)(q,x)=j(j-1)(j-2)(j-3)q^{j(j+1)/2}x^{j-2}
    &{\rm and}\\ \\ 
    (\partial ^2 A_{j-2}/\partial u^2)(q,x)=\frac{(j-2)(j-1)}{2}\cdot \left( 
    \frac{(j-2)(j-1)}{2}-1\right) \cdot q^{((j-2)(j-1)/2)-2}x^{j-2}~,&{\rm so}\\ \\
      4q^5(\partial^2 A_{j-2}/\partial u^2)(q,q^2x)=j(j-1)(j-2)(j-3)q^{j(j+1)/2}x^{j-2}
  \end{array}$$
  from which the second equality follows.
    
\end{proof}

\begin{proof}[Proof of Proposition~\ref{propsecond}]
  By the first of equalities (\ref{equ2equ}) the function
  $\partial^2 \theta /\partial x^2$ is positive for $x=-q^{-3/2}$, see
Lemma~\ref{lm1/2}. 
Suppose first that
$q>0$ is sufficiently small. Then the function $\theta (q,.)$ belongs to the
Laguerre-P\'olya class $\mathcal{LP}I$
(see \cite[part~(3) of Theorem~1]{KoDBAN14}) and its derivatives w.r.t. $x$ 
  of any order are also in $\mathcal{LP}I$, so all zeros of
  $\partial^2 \theta /\partial x^2$ are real negative.

  For $q>0$ close to $0$ and for $j=0$, $1$ and $2$,
  the rightmost of the zeros of $(\partial ^j\theta /\partial x^j)(q,.)$ are
  equivalent respectively to $-q^{-1}$, $-q^{-2}/2$ and $-q^{-3}/3<-q^{-3/2}$.
  Hence for $q>0$ sufficiently small, all zeros of
  $\partial^2 \theta /\partial x^2$ are smaller than $-q^{-3/2}$.

  We show that as $q$ increases from $0$ to $1$, the function
  $\partial^2 \theta /\partial x^2$ can lose real zeros, but not acquire such.
  This fact together with $\partial^2 \theta /\partial x^2|_{x=-q^{-3/2}}>0$ for
  $q>0$ sufficiently small 
  implies then that it has no real zeros for $x\geq -q^{-3/2}$.

  All coefficients of $\partial^2 \theta /\partial x^2$ are positive,
  so it has no positive zeros. Suppose that for some $q_*\in (0,1)$, the
  function $\partial^2 \theta /\partial x^2$ has a zero
  $x_*\in [-q_*^{-3/2},0)=:I_*$.
  Suppose that $q_*$ is the smallest value of $q$ for which this happens.
  Then this zero cannot be of odd multiplicity. Indeed, if this is the case,
  then there exists an interval $[\alpha ,\beta ]\subset I_*$,
  $\alpha <x_*<\beta$, such that the quantities 

  $$(\partial^2 \theta /\partial x^2)(q_*,\alpha )~~~\, {\rm and}~~~\,
  (\partial^2 \theta /\partial x^2)(q_*,\beta )$$
  are non-zero and of opposite signs. But then by continuity this is true also 
  for $q<q_*$ sufficiently close to $q_*$ which contradicts the minimality of
  $q_*$.

  Suppose that the zero $x_*$ is of even multiplicity; if there are several such zeros, we choose the rightmost one. Then given that
  $(\partial^2 \theta /\partial x^2)(q,0)=2q^3>0$, the
  function $\partial^2 \theta /\partial x^2$ has a local minimum at $x_*$.
  Two successive defferentiations w.r.t. $x$ of equality (\ref{equpartder}) 
  yield the equality

  $$2q\cdot \partial ^3\theta /\partial q\partial x^2=6\cdot
  \partial ^2\theta /\partial x^2
  +6x\cdot \partial ^3\theta /\partial x^3
  +x^2\cdot \partial ^4\theta /\partial x^4~.$$
  For $q<q_*$ and close to $q_*$, at the minimum of the function
  $\partial^2 \theta /\partial x^2$ on the interval $I_*$ one has

  $$\partial ^2\theta /\partial x^2>0~,~~~\, 
  \partial ^3\theta /\partial x^3=0~~~\, {\rm and}~~~\,
  \partial ^4\theta /\partial x^4\geq 0~;$$
  the latter inequality holds true in some rectangle $[q_*-\varepsilon ,q_*+\varepsilon ]\times [x_*-\varepsilon ,x_*+\varepsilon ]$. Hence $\partial ^3\theta /\partial q\partial x^2>0$, i.~e. the minimal
  value of
  $\partial ^2\theta /\partial x^2$ on $I_*$, increases as $q$ increases and one
  cannot have $(\partial ^2\theta /\partial x^2)(q_*,x_*)=0$ which
  contradiction proves the first claim of the proposition. The second claim is
  an immediate corollary of the first of equalities (\ref{equ2equ}) (for $x=0$,
  one has $\partial \theta /\partial q=0$).
  
  \end{proof}

\begin{proof}[Proof of Lemma~\ref{lmkgeq1}]
  Suppose that $k\geq 1$. One has

  $$\varphi_k'=\partial \theta /\partial q+
  (1-k)\cdot \partial \theta /\partial x|_{x=-q^{k-1}}\cdot q^{k-2}~.$$
  The first term in
  the right-hand side is negative, see Proposition~\ref{propsecond}.
  Having in mind that
  $\partial \theta /\partial x>0$ for $x\in (-q^{-1},\infty )$,
  see \cite[Part~(3) of Proposition~4.6]{KoPRSE2}, one concludes that the second
  term is non-positive for $q>0$ (negative for $k>1$ and zero for $k=1$).
  For $q=0$ and $k>1$, one deduces from
  the series of $\theta$ that $\varphi_k'(0)=0$; one finds that
  $\varphi_1'(0)=-1$.  

  For $k=1/2$, the lemma follows from $\nu '(q)<0$, see the proof of
  Lemma~\ref{lm1/2}; in this case one has
  $\varphi_{1/2}'(0)=-\infty$.
  
  \end{proof}

\begin{proof}[Proof of Proposition~\ref{propincr}]
  We use the Jacobi triple product, see (\ref{equT}), and the equality
  (\ref{equthetaG}). Set $j_0:=[a]$
  (the integer part of~$a$) and $s:=\sum_{j=1}^{j_0}(a-j)=j_0(2a-j_0-1)/2>0$. 
  Hence

  $$\begin{array}{rcl}\Theta^0:=\Theta^*(q,-q^{-a})&=&
    \prod_{j=1}^{\infty}(1-q^j)(1-q^{j-a})(1-q^{j+a-1})\\ \\
      &=&-q^{-s}\prod_{j=1}^{\infty}((1-q^j)(1-q^{j+a-1}))
      \prod_{j=j_0+1}^{\infty}(1-q^{j-a})\prod_{j=1}^{j_0}(1-q^{a-j})~.\end{array}$$
  The factors in the three products are positive and decreasing
  in $q\in (0,1)$, so the function $q^s\Theta^0$ is a minus product of
  positive and strictly decreasing (from $1$ to $0$)
  functions hence $\Theta^0$ increases
  from $-\infty$ to~$0$. At the same time $-G(q,-q^{-a})=(1-\varphi_a)(q)$
  is a function which increases
  from $0$ to $1/2$. Thus the sum $\theta (q,-q^{-a})=\Theta^0-G(q,-q^{-a})$
  strictly increases from $-\infty$ to $1/2$ and hence equals $0$
  exactly for one value of $q\in (0,1)$.

  If $a\in \mathbb{N}$, then there is a zero factor in $\Theta^0$ and
  $\theta (q,-q^{-a})=-G(q,-q^{-a})=q^a\varphi_{a+1}>0$. If $a>0$, 
  $a\not\in K^{\dagger}$, then there is an even number
  of negative factors in $\Theta^*(q,-q^{-a})$, so $\Theta^*(q,-q^{-a})>0$
  and $-G(q,-q^{-a})>0$ hence $\theta (q,-q^{-a})>0$.
  
  \end{proof}

\begin{proof}[Proof of Proposition~\ref{prop2kincrease}]
  Consider the zeros $\xi_{2k}$ and $\xi_{2k-1}$ as functions in
  $q\in (0,\tilde{q}_k]$. The union of their graphs is a smooth curve in
  the plane of the variables $(q,x)$, see \cite[Theorem~3]{KoSe}. 
Their derivatives
tend to $+\infty$ and $-\infty$ respectively as $q\rightarrow \tilde{q}_k^-$.
For $q=\tilde{q}_k$,
there exists a unique curve $x=-q^{-a_1}$, $a_1>0$, 
such that
$\theta (\tilde{q}_k,-(\tilde{q}_k)^{-a_1})=0$, see
Proposition~\ref{propincr}. As $q$ increases in a small neighbourhood of
$\tilde{q}_k$, the values of
$\theta (q,-q^{-a_1})$ pass from negative to positive.

Suppose that there
  exists $q_{\natural}\in (0,\tilde{q}_k]$ such that $\xi_{2k}'(q_{\natural})=0$.  
Choose the largest
possible such number $q_{\natural}$, so $\xi_{2k}'>0$ for $q\in (q_{\natural},\tilde{q}_k)$.
There exists a unique curve $x=-q^{-a_2}$, $a_2>0$, 
such that $\theta (q_{\natural},-(q_{\natural})^{-a_2})=0$. For $q\in (q_{\natural},\tilde{q}_k)$ close
to $q_{\natural}$, one has $-q^{-a_2}>\xi_{2k}$. But then 
as $q$ increases up from $q_{\natural}$, the values of
$\theta (q,-q^{-a_2})$ pass from positive to negative which contradicts
Proposition~\ref{propincr}. Besides, 
$\theta (q,-q^{-a_2})\rightarrow -\infty$ as $q\rightarrow 0^+$, so the curve
$x=-q^{-a_2}$ intersects the zero set of $\theta$ at least thrice, for at least
one value of $q$ larger, for at least one value of $q$ smaller than $q_{\natural}$
and for $q=q_{\natural}$.
Hence such a number $q_{\natural}$ does not exist and
one has $\xi_{2k}'>0$ on $(0,\tilde{q}_k)$.
\end{proof}


\begin{thebibliography}{40}
\bibitem{AnBe} G. E. Andrews, B. C. Berndt,  
Ramanujan's lost notebook. Part II. Springer, NY, 2009.
\bibitem{BeKi} B. C. Berndt, B. Kim, 
Asymptotic expansions of certain partial theta functions. 
Proc. Amer. Math. Soc. 139:11 (2011), 3779-3788.
\bibitem{BFM} K.~Bringmann, A.~Folsom and A.~Milas, 
Asymptotic behavior 
of partial and false theta functions arising from Jacobi forms and regularized 
characters. 
J. Math. Phys. 58:1 (2017), 011702, 19 pp.
\bibitem{BrFoRh} K. Bringmann, A. Folsom, R. C. Rhoades, 
Partial theta functions 
and mock modular forms as $q$-hypergeometric series, 
Ramanujan J. 29:1-3 (2012), 
295-310, 
http://arxiv.org/abs/1109.6560
\bibitem{CMW} T.~Creutzig, A.~Milas and S.~Wood, 
On regularised quantum 
dimensions of the singlet vertex operator algebra and false theta functions. 
Int. Math. Res. Not. 5 (2017), 1390-1432.
\bibitem{FGM} R. Flores and J. Gonz\'alez-Meneses,
  On the growth of Artin-Tits monoids and the partial theta function. 
J. Combin. Theory Ser. A 190 (2022), Paper No. 105623, 39 pp.
\bibitem{Ha} G.~H.~Hardy, 
On the zeros of a class of integral functions, 
Messenger of Mathematics, 34 (1904), 97-101.
\bibitem{Hu} J. I. Hutchinson, 
On a remarkable class of entire functions, 
Trans. Amer. Math. Soc. 25 (1923), 325-332.
\bibitem{KaLoVi} O.M. Katkova, T. Lobova and A.M. Vishnyakova, 
On power series 
having sections with only real zeros. 
Comput. Methods Funct. Theory 3:2 (2003), 
425-441.
\bibitem{Ka} B. Katsnelson, On summation of the Taylor series 
of the function 
$1/(1-z)$ by the theta summation method,
Complex analysis and dynamical systems VI. Part 2, 141–157.
Contemp. Math., 667
Israel Math. Conf. Proc.
American Mathematical Society, Providence, RI, 2016
\bibitem{KoFAA} V.P. Kostov, 
On the multiple zeros of a partial theta function, 
Funct. Anal. Appl. 
(Russian version 50, No. 2 (2016) 84-88, 
English version 50, No. 2 (2016) 153-156).
\bibitem{KoDBAN14} V.P. Kostov, A property of a partial theta function,
  Comptes Rendus Acad. Sci. 
  Bulgare 67, No. 10 (2014) 1319-1326.
\bibitem{KoDBAN15} V.P. Kostov, Asymptotic expansions of zeros of a partial
  theta function, Comptes Rendus Acad. Sci. Bulgare 68, No. 4 (2015) 419-426.
\bibitem{KoFAA19} V.P. Kostov, On the complex conjugate zeros of the partial
  theta function, Funct. 
  Anal. Appl. (English version 2019, 53:2, 149-152,
  Russian version 2019, 53:2, 87-91).
\bibitem{KoBSM1} V.P. Kostov, 
On the zeros of a partial theta function, 
Bull. 
Sci. Math. 137:8 (2013), 1018-1030.
\bibitem{KoBSM2}  V.P. Kostov, On the double zeros of a partial theta function, 
Bull. Sci. Math. 140, No. 4 (2016) 98-111.
\bibitem{KoPRSE2}V.P. Kostov,  
On a partial theta function and its spectrum, 
Proc. Royal Soc. Edinb. A 146:3 (2016), 609-623.
\bibitem{KoPMD} V.P. Kostov, 
A domain containing all zeros of the partial 
theta function, 
Publ. Math. Debrecen 93:1-2 (2018), 189-203.
\bibitem{KoAM} V.P. Kostov, 
A separation in modulus property of the 
zeros of a partial theta function, 
Analysis Mathematica 44:4 (2018), 
501-519.
\bibitem{KoSe}V.P. Kostov, On the zero set of the partial theta function,
  Serdica Math. J. 45 (2019), 225-258.
\bibitem{KoMatStud} V.P. Kostov, A domain free of the zeros of the
  partial theta function.
  Mat. Stud. 58 (2022), no. 2, 142-158.
\bibitem{KoAnn24} V.P. Kostov, No zeros of the partial theta function in the
  unit disk, Annual of Sofia University ``St. Kliment Ohridski'', Faculty
  of Mathematics and Informatics 111 (2024) 129-137.
  DOI: 10.60063/gsu.fmi.111.129-137
\bibitem{KoArxiv} V.P. Kostov, On the location of the complex conjugate zeros
  of the partial theta function, Serdica Math. J. (to appear) arXiv:2501.15866. 
\bibitem{KoSh} V.P. Kostov and B. Shapiro, 
Hardy-Petrovitch-Hutchinson's 
problem and partial theta function, 
Duke Math. J. 162:5 (2013), 825-861. 
\bibitem{LuSa} D.S. Lubinsky, E. Saff, Convergence of Pad\'e approximants
  of partial theta functions and the Rogers-Szeg\H{o}
polynomials, Constructive Approximation, 3 (1987), 331-361.
\bibitem{EM} E.T. Mortenson, On the dual nature of partial theta functions
  and Appell-Lerch sums, Adv. Math. 264 (2014), 236-260.
\bibitem{Ost} I.~V.~Ostrovskii, 
On zero distribution of sections and tails 
of power series, 
Israel Math. Conf. Proceedings, 15 (2001), 297-310.
\bibitem{Pe} M.~Petrovitch, 
Une classe remarquable de s\'eries enti\`eres, 
Atti del IV Congresso Internationale dei Matematici, Rome (Ser. 1), 2 
(1908), 36-43. 
\bibitem{PoSz} G. P\'{o}lya, G. Szeg\H{o}, 
  Problems and Theorems in Analysis, Vol. 1, Springer, Heidelberg 1976.
\bibitem{Pr} T.~Prellberg, The combinatorics of the leading root of the 
partial theta function, http://arxiv.org/pdf/1210.0095.pdf
\bibitem{So} A.~Sokal, 
The leading root of the partial theta function, 
Adv. Math. 229:5 (2012), 2603-2621. 
arXiv:1106.1003.
\bibitem{Sun} L. H. Sun, An extension of the Andrews-Warnaar partial theta
  function identity. 
Adv. in Appl. Math. 115 (2020), 101985, 20 pp.
\bibitem{WPG} B.~Walter, G.~Perfetto and A.~Gambassi,
  Thermodynamic phases in first detected return times
  of quantum many-body systems, Physical Review A 111, L040202 (2025).
\bibitem{WM} J. Wang and X. Ma, On the Andrews-Warnaar identities
  for partial theta functions. Adv. in Appl. Math. 97 (2018), 36-53.
\bibitem{Wa} S. O. Warnaar, 
Partial theta functions. I. Beyond 
the lost notebook, 
Proc. London Math. Soc. (3) 87:2 (2003), 363-395.
\bibitem{Wei} C. Wei, Partial theta function identities from Wang and Ma's
  conjecture.
J. Difference Equ. Appl. 26 (2020), no. 4, 532-539.
\end{thebibliography}
\end{document}